\documentclass[10pt]{smfart}

\RequirePackage[T1]{fontenc}
\RequirePackage{amsfonts,latexsym,amssymb}
\RequirePackage[frenchb]{babel}
\addto\extrasfrenchb{\bbl@nonfrenchitemize
\bbl@nonfrenchspacing}

\RequirePackage{mathrsfs}
\let\mathcal\mathscr

\theoremstyle{plain}
\numberwithin{equation}{section}
\newtheorem{prop}[equation]{\propname}
\newtheorem{theo}[equation]{\theoname}
\newtheorem{conj}[equation]{\conjname}
\newtheorem{coro}[equation]{\coroname}

\newtheorem{lemm}[equation]{\lemmname}
\theoremstyle{definition}
\theoremstyle{remark}

\newtheorem{rema}[equation]{\remaname}
\newtheorem{exem}[equation]{\exemname}

\newtheorem{ques}[equation]{Question}

\newcommand{\mar}[1]{\marginpar{\tiny #1}}

\usepackage[matrix,arrow]{xy}
\usepackage{url}
\usepackage{fge}
\newcommand{\moins}{\mathbin{\fgebackslash}}

\usepackage[pagebackref]{hyperref}

\let\cal\mathcal
\let\goth\mathfrak

\def\wotimes{\hskip.5mm\widehat\otimes\hskip.5mm}
\def\wGamma{\widehat\Gamma}
\def\oGamma{{\overline\Gamma}}
\def\rg{{\rm R}\Gamma}

\let\emptyset\varnothing
\def\GG{{\mathbb G}} \def\BB{{\mathbb B}} \def\PP{{\mathbb P}} \def\UU{{\mathbb U}}
\def\LL{{\mathbb L}}
\def\HH{{\mathbb H}}
\newcommand{\eet}{\operatorname{\acute{e}t} }

\def\demo{\noindent{\itshape D\'emonstration}.\ ---\ \ignorespaces}
\def\Qbar{\overline{\bf Q}}
\def\Q{{\bf Q}} \def\Z{{\bf Z}}
\def\C{{\bf C}}
\def\N{{\bf N}}
\def\R{{\bf R}}
\def\O{{\cal O}}
\def\dual{{\boldsymbol *}}
\def\bmu{{\boldsymbol\mu}}

\def\Q{{\bf Q}} \def\Z{{\bf Z}}
\def\C{{\bf C}}
\def\N{{\bf N}}
\def\O{{\cal O}} 

\def\Qbar{{\overline{{\bf Q}}}}

\def\epsilon{\varepsilon}

\def\wZ{{\widehat\Z}}
\def\wZp{{\wZ^{]p[}}}

\def\tt{\widetilde\theta}

\def\G{{\cal G}}

\def \R{{\bf R}}               \def \N{{\bf N}}

\makeatletter
\let\over\@@over
\makeatother

\def\piqp{{{\bf P}^1}}

 \def\A{{\bf A}} \def\B{{\bf B}}

\def\Ai{\A^{]\infty[}}
\def\Aidu{\A^{]\infty[,\dual}}

\def\matrice#1#2#3#4{{\big(\begin{smallmatrix}#1&#2\\ #3&#4\end{smallmatrix}\big)}}

\begin{document}
\title
{Exercices ad\'eliques}
\author{Pierre Colmez}
\address{C.N.R.S., IMJ-PRG, Sorbonne Universit\'e, 4 place Jussieu,
75005 Paris, France}
\email{pierre.colmez@imj-prg.fr}
\begin{abstract}
On \'etudie la cohomologie du groupe des points rationnels d'un groupe
alg\'ebrique lin\'eaire sur des espaces fonctionnels ad\'eliques,
et on explore le lien avec la cohomologie compl\'et\'ee d'Emerton.
\end{abstract}
\begin{altabstract}
We study the cohomology of the group of rational points of a linear algebraic group
on adelic functional spaces, and we explore the relation with Emerton's completed cohomology.
\end{altabstract}
\setcounter{tocdepth}{1}

\maketitle

\stepcounter{tocdepth}
{\Small
\tableofcontents
}

\section*{Introduction}
Dans l'avatar $p$-adique du programme de Langlands, un objet fondamental est la cohomologie
compl\'et\'ee d'Emerton~\cite{Em06,CE}: si $\GG$ est un groupe r\'eductif d\'efini sur $\Q$,
et si $K_\A$ est un sous-groupe ouvert compact maximal\footnote{On note $\A$
l'anneau des ad\`eles de $\Q$, et $\A^{]S[}$ l'anneau des ad\`eles hors de $S$, si
$S$ est un ensemble fini de places de $\Q$.} de $\GG(\A)$, on dispose
de la tour des espaces localement sym\'etriques
$Y(K):=\GG(\Q)\backslash\GG(\Q)/K$, pour $K$ sous-groupe ouvert de $K_\A$,
et on pose
$$(\widehat H^i)^{(p)}:=\varinjlim\nolimits_{K^{]p[}}\widehat H^i(K^{]p[}),
\quad \widehat H^i(K^{]p[})
:=\Q_p\otimes_{\Z_p}\big({\varprojlim\nolimits_k}\big({\varinjlim\nolimits_{K_p}}
\hskip-1mm H^i(Y(K^{]p[}K_p),\Z/p^k)\big)\big)$$
o\`u $K_p$ (resp.~$K^{]p[}$) d\'ecrit les sous-groupes ouverts de $\GG(\Q_p)$ (resp.~$\GG(\A^{]p[})$)
tels que $K^{]p[}K_p\subset K_\A$. On d\'efinit de m\^eme la cohomologie compl\'et\'ee
\`a support compact, et les homologie compl\'et\'ee et homologie compl\'et\'ee \`a support
compact.  Tous ces groupes sont naturellement munis d'une action continue de $K_\A$ qui
s'\'etend naturellement en une action continue de $\GG(\A)$ et l'un des buts du programme
de Langlands $p$-adique est de d\'ecrire les repr\'esentations de $\GG(\A)$ apparaissant
dans ces groupes.

On peut utiliser le lemme de Shapiro pour transformer la cohomologie de $Y(K)$ en celle
de $Y(K_\A)$ \`a valeurs dans un syst\`eme de coefficents (fonctions sur $K_\A/K$),
et passer \`a la limite pour \'ecrire $\widehat H^i(K^{]p[})$ comme la cohomologie
de $Y(K_\A)$ \`a valeurs dans un gros syst\`eme de coefficients 
(espace ${\cal C}^{(p)}(K_\A)$ des fonctions continues
sur $K_\A$, lisses en dehors de $p$); 
c'est le point de vue de Hill~\cite{hill}.

D'un autre c\^ot\'e, dans le programme de Langlands classique, le point
de d\'epart est 
plut\^ot les formes automorphes, i.e.~les invariants sous $\GG(\Q)$ 
dans des espaces fonctionnels sur $\GG(\A)$:
on fait agir $\GG(\Q)$ et $\GG(\A)$ sur les fonctions
sur $\GG(\A)$ par
$$(\gamma *\phi)(x)=\phi(\gamma^{-1}x),\ {\text {si $\gamma\in \GG(\Q)$}},
\quad
(g\star\phi)(x)=\phi(xg),\ {\text {si $g\in \GG(\A)$}}.$$
Ces deux actions commutent.  Il s'ensuit que, si $X$ est un espace de fonctions
sur $\GG(\A)$ qui est stable par ces deux actions, les groupes de cohomologie
$H^i(\GG(\Q),X)$ sont munis d'une action de $\GG(\A)$.

Dans le cas classique, on prend en g\'en\'eral $i=0$,
on fixe un caract\`ere central,
et on prend pour $X$ l'espace des fonctions de carr\'e
int\'egrable modulo le centre ou bien son sous-espace des formes automorphes 
(fonctions $K_\A$-finies, ${\cal C}^\infty$ comme fonctions
de $x_\infty$, se transformant par un caract\`ere sous l'action du centre
de l'alg\`ebre enveloppante de l'alg\`ebre de Lie,
et \`a d\'ecroissance rapide \`a l'infini).

Si on essaye d'interpr\'eter la cohomologie compl\'et\'ee
d'Emerton de cette mani\`ere, on peut transformer la cohomologie
de $Y(K_\A)$ en celle de son groupe fondamental\footnote{C'est le point de vue adopt\'e
dans~\cite{bbk-Kato,CW} (pour $\GG={\bf GL}_2$).}, \`a savoir le groupe
arithm\'etique $\Gamma:=\GG(\Q)\cap K_\A^{]\infty[}$ (comme $Y(K_\A)$ peut avoir
plusieurs composantes connexes, il y a en fait plusieurs groupes arithm\'etiques
qui interviennent), et on peut utiliser le lemme de Shapiro pour passer de
$\Gamma$ \`a $\GG(\Q)$, ce qui transforme ${\cal C}^{(p)}(K_\A)$ en un espace
proche de ${\cal C}^{(p)}(\GG(\A))$ et sugg\`ere qu'un 
espace naturel \`a regarder est $H^i(\GG(\Q),{\cal C}^{(p)}(\GG(\A)))$.
Mais l'espace ${\cal C}^{(p)}$ n'est pas le seul espace naturel
que l'on peut consid\'erer; 
parmi les espaces possibles, mentionnons les espaces
${\cal C}\supset{\cal C}^{(p)}\supset{\rm LA}\supset{\rm LP}$, o\`u ayant fix\'e
une extension finie $L$ de $\Q_p$, on pose:
\begin{align*}
{\cal C}(\GG(\A))&:=\{\phi:\GG(\A)\to L,\ {\text{$\phi$ continue}}\}\\
{\cal C}^{(p)}(\GG(\A))&:=\{\phi\in {\cal C}(\GG(\A)),
\ {\text{$\phi$ localement lisse pour l'action de $\GG(\wZ^{]p[})$}}\}\\
{\rm LA}(\GG(\A))&:=\{\phi\in {\cal C}^{(p)}(\GG(\A)),
\ {\text{$\phi$ localement analytique en $x_p$}}\}\\
{\rm LP}(\GG(\A))&:=\{\phi\in {\cal C}^{(p)}(\GG(\A)),
\ {\text{$\phi$ localement alg\'ebrique en $x_p$}}\}
\end{align*}
Notons que, $L$ \'etant totalement discontinue,
une fonction continue $\phi:\GG(\A)\to L$ est constante modulo la composante connexe
$\GG(\R)_+$ de l'identit\'e de $\GG(\R)$. Si on impose \`a $\phi$ d'\^etre aussi invariante par
translation \`a gauche par $\GG(\Q)$, cela force $\phi$ \`a se factoriser
par le quotient \`a gauche par l'adh\'erence de $\GG(\Q)\GG(\R)_+$,
et comme ce quotient est tr\`es petit, le groupe
$H^0(\GG(\Q),{\cal C}(\GG(\A)))$ n'est pas tr\`es int\'eressant.
On est donc forc\'e de regarder les
$$H^i(\GG,X):=H^i(\GG(\Q),X(\GG(\A)))$$
 pour $i\geq 1$, dans le cadre du programme
de Langlands $p$-adique. 

Dans cet article, nous explorons les liens entre ces diff\'erents espaces et leurs versions
\`a support compact, purement du point de vue des repr\'esentations de $\GG(\A)$
(i.e.~on ne s'int\'eresse qu'au versant automorphe du programme de Langlands $p$-adique).
Nous n'abordons pas l'aspect le plus int\'eressant, \`a savoir le lien avec les repr\'esentations
du groupe de Galois absolu ${\rm Gal}_\Q$ de $\Q$.  Les r\'esultats d'Emerton~\cite{Em08}
(cf.~aussi~\cite{CW}) pour
$\GG={\bf GL}_2$ sugg\`erent que le centre de ${\rm End}_{\GG(\A)}H^i(\GG,{\cal C})$
param\`etre des pseudocaract\`eres de ${\rm Gal}_\Q$, mais il est probablement pr\'ematur\'e
de formuler une conjecture pr\'ecise (nous renvoyons \`a~\cite[\S\,1.8]{CE} pour des
r\'eflexions sur ce th\`eme).

Le gros de l'article est consacr\'e au cas o\`u $\GG={\bf GL}_2$; beaucoup des r\'esultats
sont des reformulations en langage ad\'elique de r\'esultats d'Emerton (d'o\`u le
titre de l'article).  Supposons donc $\GG={\bf GL}_2$ dans le reste de cette introduction
(sauf mention explicite du contraire).

\begin{theo}
{\rm (i)} Si $X={\cal C}, {\cal C}^{(p)}, {\rm LA}, {\rm LP}, {\rm LC}$,
alors $H^i(\GG,X)=0$ si $i\geq 2$, et $H^0(\GG,X)\cong X(\wZ^\dual)$, l'action de $\GG(\A)$
sur $\wZ^\dual\cong \A^\dual/\R_+^\dual\Q^\dual$ se faisant \`a travers $\det:\GG(\A)\to\A^\dual$.

{\rm (ii)} $H^1(\GG,{\cal C}^{(p)})$ est la limite inductive des
$H^1(\GG(\Q),{\cal C}(\GG(\A)/K^{]p[}))$ qui sont
des repr\'esentations admissibles de $\GG(\Q_p)$.

{\rm (iii)} $H^1(\GG,{\cal C}^{(p)})$ est l'ensemble des vecteurs
$\GG(\wZp)$-lisses de $H^1(\GG,{\cal C})$.

{\rm (iv)} $H^1(\GG,{\rm LA})$ est l'ensemble des vecteurs
localement analytiques de $H^1(\GG,{\cal C}^{(p)})$.

{\rm (v)} $H^1(\GG,{\rm LP})$ est l'ensemble des vecteurs
localement alg\'ebriques de $H^1(\GG,{\cal C}^{(p)})$.
\end{theo}

\begin{rema}
Si $\GG$ est g\'en\'eral:

{\rm (i)}  
 Calegari et Emerton~\cite[conj.\,1.5]{CE} 
conjecturent que\footnote{Les notations $q_0$ et $\ell_0$
sont standard: $\ell_0$ mesure le d\'efaut d'existence de s\'eries discr\`etes pour $\GG(\R)$
et $q_0=(d-\ell_0)/2$ o\`u $d$ est la dimension des $Y(K)$.}
$H^i(\GG,{\cal C}^{(p)})=0$ si $i>q_0$ et est petit si $i<q_0$.

{\rm (ii)} Le (ii) reste valable pour tout $H^i$; 
il est possible que le (iii) reste valable mais dans
le cas o\`u $\GG={\rm Res}_{F/\Q}{\bf GL}_1$ c'est \'equivalent \`a la conjecture
de Leopoldt pour $F$ (nous prouvons que $H^i(\GG,{\cal C})=0$ si $i\geq 1$,
et Hill~\cite{hill} a prouv\'e que $H^i(\GG,{\cal C}^{(p)})=0$ pour tout $i\geq 1$ si et seulement
si la conjecture de Leopoldt est vraie pour $F$).
Le (iv) reste vrai pour tout $i$, et ce que l'on peut esp\'erer concernant le (v)
n'est pas compl\`etement clair. 
\end{rema}

En ce qui concerne la cohomologie \`a support compact, nous prouvons les r\'esultats suivants
(toujours pour $\GG={\bf GL}_2$).
\begin{theo}
{\rm (i)}
On a $H^i_c(\GG,{\cal C})=0$ si $i\neq 1$, et
une suite exacte de $\GG(\A)$-modules
$$0\to {\cal C}(\wZ^\dual)\to{\rm Ind}_{\BB(\A)}^{\GG(\A)}{\cal C}(\LL(\wZ))\to
H^1_c(\GG,{\cal C})\to H^1(\GG,{\cal C})\to 0$$
o\`u $\wZ^\dual=\A^\dual/\R_+^\dual\Q^\dual$ sur lequel $\GG(\A)$ agit \`a travers le d\'eterminant,
et $\BB(\A)$ \`a travers $\LL(\A)$ sur $\LL(\wZ)$.

{\rm (ii)} $H^i_c(\GG,{\rm LP})=0$ si $i\neq 1$;
$H^1_c(\GG,{\rm LP})$ est l'ensemble des vecteurs localement alg\'ebriques
de $H^1_c(\GG,{\cal C})$ et on a une suite exacte de $\GG(\A)$-modules
$$\xymatrix@R=2mm@C=4mm{
0\ar[r]& {\rm LP}(\wZ^\dual)\ar[r]&{\rm Ind}_{\BB(\A)}^{\GG(\A)}{\rm LP}(\LL(\wZ))\ar[r]&
H^1_c(\GG,{\rm LP})\ar[dl]\\
& & H^1(\GG,{\rm LP})\ar[r]&
\big({\rm Ind}_{\BB(\A)}^{\GG(\A)}{\rm LP}(\LL(\wZ))\big)^\vee
\ar[r]& {\rm LP}(\wZ^\dual)\ar[r]& 0
}$$
o\`u $\BB$ est le borel sup\'erieur de $\GG$, $\LL$ est son l\'evi, et $(-)^\vee$
d\'esigne la contragr\'ediente.
\end{theo}

\section{Pr\'eliminaires}\label{chapi1}

Le but de ce chapitre est de fixer un certain nombre de notations et de normalisations.
\Subsection{Ad\`eles}\label{prelim1}
\subsubsection{Notations}\label{prelim2}
On note ${\cal P}$ \index{premiers@\premiers}l'ensemble des nombres premiers.  
%Si $S\subset{\cal P}$, on dit qu'un entier $N$ est {\it \`a support dans $S$} si tous ses diviseurs
%premiers appartiennent \`a $S$.

Si $v\in {\cal P}\cup\{\infty\}$
est une place de $\Q$, on note $\Q_v$ le compl\'et\'e de $\Q$ en $v$ (et donc $\Q_\infty=\R$).
On note $\A$ l'anneau des \index{adeles@\adeles}ad\`eles de $\Q$, produit restreint des $\Q_v$:
si $\wZ=\prod_{p\in{\cal P}}\Z_p$ est le compl\'et\'e profini de $\Z$, on a
$\A=\R\times(\Q\otimes\widehat\Z)$.

Si $S\subset {\cal P}\cup\{\infty\}$ est fini, \index{Q@\QQQ}on \index{Z@\ZZZ}pose
$$\Z_S=\prod_{\ell\in S\cap{\cal P}}\Z_\ell, 
\quad\Z[\tfrac{1}{S}]=\Z[\tfrac{1}{\ell},\,\ell\in S\cap{\cal P}],
\quad \Q_S=\prod_{v\in S}\Q_v$$
On note $\A^{]S[}$ le produit restreint des $\Q_v$ pour $v\notin S$ et $\wZ^{]S[}$ son
intersection avec $\wZ$.
Par exemple, $\A^{]\infty[}=\Q\otimes\wZ$. On a $\A=\Q_S\times \A^{]S[}$ 
et $\wZ=\Z_S\times\wZ^{]S[}$ pour tout $S$.

Si $x$ est un objet ad\'elique (i.e., un \'el\'ement de $\A$, $\A^\dual$, ${\bf GL}_n(\A)$,
etc.), et si 
$v$ est une place de $\Q$, on note $x_v$ la composante de $x$ en $v$ (et donc $x=(x_v)_v$).
Si
$S$ est un ensemble fini de places de $\Q$, on note $x_S=(x_v)_{v\in S}$ la composante de $x$ 
en les \'el\'ements de $S$ et
$x^{]S[}=(x_v)_{v\notin S}$ la composante de $x$ hors de $S$ (et donc $x=(x_S,x^{]S[})$.

Si $\GG$ est un groupe alg\'ebrique sur $\Q$, les projections $\GG(\A)\to \GG(\Q_v)$ ont
des sections naturelles qui permettent de consid\'erer les
$\GG(\Q_v)$ comme des sous-groupes de $\GG(\A)$.

De m\^eme, si $S$ est un ensemble fini de places de $\Q$, alors $\GG(\Q_S)=\prod_{v\in S}\GG(\Q_v)$
est naturellement un sous-groupe de $\GG(\A)$ et tout \'el\'ement de $\GG(\A)$ peut
s'\'ecrire de mani\`ere unique sous la forme $x_Sx^{]S[}$ avec $x_S\in \GG(\Q_S)$ et
$x^{]S[}\in \GG(\A^{]S[})$; de plus $x_S$ et $x^{]S[}$ commutent.

%{\tiny
%\subsubsection{Caract\`eres additifs}\label{prelim3}
%On \index{expo@\exponent}note ${\bf e}_\infty:\C\to\C^\dual$ le caract\`ere du groupe additif
%$${\bf e}_\infty(\tau)=e^{-2i\pi\,\tau}.$$
%On note ${\bf e}_\A:\A\to\C^\dual$ le caract\`ere du groupe additif
%se factorisant \`a travers $\A/\Q$ et dont la restriction \`a $\R\subset\A$
%est ${\bf e}_\infty$.
%Si $\ell$ est un nombre premier, on note ${\bf e}_\ell$ la restriction de ${\bf e}_\A$
%\`a $\Q_\ell\subset \A$.  Alors ${\bf e}_\ell$ se factorise \`a travers $\Q_\ell/\Z_\ell=
%\Z[\frac{1}{\ell}]/\Z\subset \R/\Z$ et on a
%${\bf e}_\ell(x_\ell)=e^{2i\pi\,\overline x_\ell}$, o\`u $\overline x_\ell$ est
%l'image de $x$ dans $\R/\Z$ par l'inclusion ci-dessus.
%On a alors
%$${\bf e}_\A((x_v)_v)=\prod_v{\bf e}_v(x_v).$$

\Subsubsection{Le caract\`ere $|\ |_\A$ et ses avatars}\label{prelim4}
\noindent $\bullet$ {\it Le caract\`ere norme $|\ |_\A$}.
Si $v$ est une place de $\Q$, on note $|\ |_v$ 
\index{adeles@\adeles}la norme sur $\Q_v$ (si $v$ est la place
correspondant \`a un nombre premier $\ell$, on a $|\ell|_v=\ell^{-1}$).
On d\'efinit $|\ |_\A$ par la formule
$$|x|_\A=\prod_v|x_v|_v.$$
\vskip.1cm
\noindent $\bullet$ {\it Le caract\`ere $\delta_\A $}.
On note $\delta_\A :\A^\dual\to \C^\dual$ le caract\`ere
$$x\mapsto \delta_\A (x)=x_{\infty}^{-1}|x|_\A,$$ 
et donc $\delta_\A $ est localement constant, \`a valeurs
dans $\Q^\dual$.
De plus,
$$\delta_{\A,\infty}={\rm sign},\quad \delta_{\A,\ell}=|\ |_\ell,
\quad
\delta_\A =|\ |_\A\ {\text{sur $(\A^{]\infty[})^\dual$}}.$$ 
\noindent $\bullet$ {\it Le caract\`ere $|\ |_{\A,p}$}.
On d\'efinit $|\ |_{\A,p}$ par $|\ |_{\A,p}(x):=x_p\delta_\A(x)$.
La formule du produit implique que $|\ |_{\A,p}$ 
se factorise \`a travers $\A^\dual/\R_+^\dual\Q^\dual$.
Notons que l'inclusion $\wZ^\dual\hookrightarrow \A^\dual$ induit un
isomorphisme $\wZ^\dual\overset{\sim}{\to} \A^\dual/\R_+^\dual\Q^\dual$.

\Subsection{Espaces fonctionnels ad\'eliques}\label{fma3}
Soit $\GG$ un groupe alg\'ebrique r\'eductif d\'efini sur $\Q$.
Alors $\GG$ a un mod\`ele lisse sur $\Z[\frac{1}{S}]$, pour $S$ assez grand.
On choisit pour tout $p\in{\cal P}$ un sous-groupe ouvert compact maximal de $\GG(\Q_p)$
\'egal \`a $\GG(\Z_p)$ si $p\notin S$, et on note encore $\GG(\Z_p)$ ce sous-groupe si $p\notin S$.
On pose $\GG(\wZ)=\prod_p\GG(\Z_p)$, et donc $\GG(\wZ)$ 
est un sous-groupe ouvert compact de $G(\A^{]\infty[})$.

\subsubsection{Fonctions continues}\label{fma6.1}
Soit $L$ une extension finie de $\Q_p$, et soit $\O_L$ l'anneau de ses entiers.
Si $\Lambda=L,\O_L$,
on \index{C@\calC}note ${\cal C}(\GG(\A),\Lambda)$ l'espace des fonctions continues
sur $\GG(\A)$ \`a valeurs dans $\Lambda$.  
On munit ${\cal C}(\GG(\A),\Lambda)$ des actions de $\GG(\Q)$
et de $\GG(\A)$ de l'introduction:
$$(\gamma*\phi)(x)=\phi(\gamma^{-1}x),\ {\text{si $\gamma\in \GG(\Q)$}},\quad
(g\star\phi)(x)=\phi(xg),\ {\text{si $g\in \GG(\A)$}}.$$
Ces actions commutent.

On note ${\cal C}^{(p)}(\GG(\A),\Lambda)$ 
le sous-espace de ${\cal C}(\GG(\A),\Lambda)$ des $\phi$ 
qui sont localement $\GG(\wZp)$-lisses. 

Dans la suite, on note simplement $X(\GG(\A))$ l'espace $X(\GG(\A),L)$, si $X$ est
une classe de fonctions.
%(i.e.~telles qu'il existe un sous-groupe ouvert
%$K_\phi$ de $\GG(\wZp)$ tel que l'on ait $\phi(x\kappa)=\phi(x)$ pour tous
%$x\in \GG(\A)$ et $\kappa\in K_\phi$): par d\'efinition,
%Par d\'efinition,
%$${\cal C}^{(p)}_u(\GG(\A))=\varinjlim_K{\cal C}(\GG(\A)/K)$$
%o\`u $K$ d\'ecrit les sous-groupes ouverts de $\GG(\wZp)$. 
%Un tel $K$ contient $\prod_{\ell\not S}\GG(\Z_\ell)$ si $S$ est assez grand,
%et donc ${\cal C}(\GG(\A)/K)$ est muni d'une action des alg\`ebres de Hecke 
%${\rm LC}(\GG(\Z_\ell)\backslash \GG(\Q_\ell)/\GG(\Z_\ell))$, pour $\ell\notin S$.

\begin{rema}\phantomsection\label{coco3}
(i)
Si $\phi\in{\cal C}(\GG(\A))$, alors $\phi$ est constante
sur les classes (\`a gauche ou \`a droite) modulo $\GG(\R)_+$ car
$\GG(\R)_+$ est connexe alors que $L$ est totalement discontinu.
Par contre, si $\ell\neq p$, une fonction continue de $\GG(\Q_\ell)$ dans $L$
n'est pas forc\'ement localement constante bien que les topologies
de $\GG(\Q_\ell)$ et $L$ ne soient pas tr\`es compatibles, et 
${\cal C}^{(p)}(\GG(\A))$ est
un (tr\`es) petit sous-espace de ${\cal C}(\GG(\A))$.

(ii) On peut aussi consid\'erer ${\cal C}_b(\GG(\A))$, espace des fonctions continues born\'ees,
ou l'espace ${\cal C}_u^{(p)}(\GG(\A))$ des fonctions $\GG(\wZp)$-lisses,
(i.e.~telles qu'il existe un sous-groupe ouvert
$K_\phi$ de $\GG(\wZp)$ tel que l'on ait $\phi(x\kappa)=\phi(x)$ pour tous
$x\in \GG(\A)$ et $\kappa\in K_\phi$); c'est la limite inductive des
${\cal C}(\GG(\A)/K)$,
o\`u $K$ d\'ecrit les sous-groupes ouverts de $\GG(\wZp)$.
\end{rema}

%\subsubsection{Mesures}
%Si $\Lambda$ est totalement discontinu,
%on note
%${\cal C}_c(\GG(\A),\Lambda)$ \index{C@\calC}l'espace des fonctions continues \`a support compact modulo $\GG(\R)_+$,
%et ${\rm Mes}(\GG(\A),\Lambda)$ \index{mesures@\mesure}son $\Lambda$-dual topologique 
%(l'espace des {\it mesures sur $\GG(\A)$
%\`a valeurs dans $\Lambda$}). On munit ${\cal C}_c(\GG(\A),\Lambda)$ des actions de $\GG(\Q)$
%et de $\GG(\A)$ sur ${\cal C}(\GG(\A),\Lambda)$, et ${\rm Mes}(\GG(\A),\Lambda)$ des actions qui s'en 
%d\'eduisent par dualit\'e.
\subsubsection{Fonctions alg\'ebriques}\label{fma4}
On suppose que $L$ est assez grand pour que les repr\'esentations
alg\'ebriques irr\'eductibles de $\GG$ soient d\'efinies sur $L$.
On note ${\rm Irr}(\GG)$ l'ensemble des $L$-repr\'esentations
alg\'ebriques de $\GG$ (\`a isomorphisme pr\`es).
On voit $W\in {\rm Irr}(\GG)$ comme une repr\'esentation de $\GG(L)$

On note ${\rm Alg}(\GG)$ \index{alg@\rmalg}l'espace des fonctions alg\'ebriques
sur $\GG(L)$, \`a valeurs dans $L$ (i.e.~$L\otimes\O(\GG)$, o\`u $\O(\GG)$ d\'esigne
l'anneau des fonctions r\'eguli\`eres sur la $\Q$-vari\'et\'e alg\'ebrique $\GG$).
On fait agir $\GG(\Q)\times \GG(L)$ sur ${\rm Alg}(\GG)$ par
$$(h_1,h_2)\cdot\phi(g)=\phi(h_1^{-1}gh_2).$$

Si $W$ est une repr\'esentation alg\'ebrique de $\GG$, on note $W^\dual$ la repr\'esentation duale.
Si $v\in W$ et $\check v\in W^\dual$, la fonction 
$$
g\mapsto\phi_{\check v,v}(g)
=\langle g\cdot\check v,v\rangle
$$
est un \'el\'ement de ${\rm Alg}(\GG)$, et l'application 
$\check v\otimes v\mapsto \phi_{\check v,v}$, 
de $W^\dual\otimes W$ dans ${\rm Alg}(\GG)$, 
est $\GG(\Q)\times \GG(L)$-\'equivariante car $\langle h_1^{-1}gh_2\cdot\check v,v\rangle
=\langle gh_2\cdot\check v,h_1\cdot v\rangle$ (et donc $\GG(\Q)$ agit trivialement sur $W^\dual$,
tandis que $\GG(L)$ agit trivialement sur $W$).  
Ceci fournit un isomorphisme $\GG(\Q)\times \GG(L)$-\'equivariant
\begin{equation}\phantomsection\label{decomp}
\bigoplus\nolimits_{W\in {\rm Irr}(\GG)}W^\dual\otimes W\cong {\rm Alg}(\GG).
\end{equation}

\subsubsection{Fonctions localement alg\'ebriques}\label{fma5}
Si $\Lambda=L,\O_L$,
soit ${\rm LC}(\GG(\A),\Lambda)$ 
l'espace des $\phi:\GG(\A)\to \Lambda$, localement constantes.
%(resp.~telles
%qu'il existe un sous-groupe ouvert $K_\phi$ de $\GG(\A^{]\infty[})$ tel que
%$\phi(g\kappa)=\phi(g)$, si $\kappa\in K_\phi$). L'espace
%${\rm LC}_u(\GG(\A),\Lambda)$ est donc l'espace des vecteurs $\GG(\wZ)$-lisses de 
%${\rm LC}(\GG(\A),\Lambda)$ (le $u$ en indice signifie "uniforme");
%on a aussi 
%$${\rm LC}_u(\GG(\A),\Lambda)=\varinjlim_K{\rm LC}(\GG(\A)/K,\Lambda)$$
 %o\`u $K$ d\'ecrit les sous-groupes ouverts de $\GG(\wZ)$.
Cet espace est stable par les actions de $\GG(\Q)$ et $\GG(\A)$ du \no\ref{fma3}.
%et l'action de $\GG(\A)$ est lisse sur ${\rm LC}_u(\GG(\A),\Lambda)$.
Soit
\begin{align*}
{\rm LP}(\GG(\A))&:={\rm LC}(\GG(\A),L)\otimes_L {\rm Alg}(\GG)
\cong\bigoplus\nolimits_{W\in{\rm Irr}}
{\rm LC}(\GG(\A))\otimes W^\dual\otimes W
%{\rm LP}_u(\GG(\A),L)&:={\rm LC}_u(\GG(\A),L)\otimes_L {\rm Alg}(\GG,L)=\oplus_{W\in{\rm Irr}}
%{\rm LC}_u(\GG(\A))\otimes W^\dual\otimes W
\end{align*}
Cet espace est muni d'actions de $\GG(\Q)$ et $\GG(\A)$ (actions diagonales 
des actions d\'efinies aux ${\rm n}^{\rm os}$~\ref{fma6.1} et~\ref{fma4}; 
dans la d\'ecomposition faisant intervenir les $W$,
$\GG(\Q)$ agit diagonalement sur ${\rm LC}$ et $W$ et trivialement sur $W^\dual$, et
$\GG(\A)$ agit diagonalement sur ${\rm LC}$ et $W^\dual$ (sur lequel il agit \`a travers son
quotient $\GG(\Q_p)$) et trivialement sur $W$.

%On a encore ${\rm LP}_u(\GG(\A),\Lambda)=\varinjlim_K{\rm LP}(\GG(\A)/K,\Lambda)$, si 
%${\rm LP}(\GG(\A)/K,\Lambda)$ d\'esigne l'espace des fonctions alg\'ebriques sur $gK$, pour tout
%$g\in\GG(\A)$ (ou, ce qui revient au m\^eme, un syst\`eme de repr\'esentants modulo~$K$).
De plus, l'application $\phi\otimes P\mapsto \phi P$ induit une injection
$\GG(\Q)\times \GG(\A)$-\'equivariante
$${\rm LP}(\GG(\A))\hookrightarrow{\cal C}(\GG(\A))$$ 

\subsubsection{Fonctions localement analytiques}
On note ${\rm LA}(\GG(\A))$ le sous-espace de ${\cal C}^{(p)}(\GG(\A))$ des $\phi$
localement analytiques en $x_p$
% et ${\rm LA}_u(\GG(\A))]subset {\rm LA}(\GG(\A))$ l'espace
%des fonctions uniform\'ement localement analytiques (i.e. telles qu'il existe
%un sous-groupe ouvert $K_\phi$ de $\GG(\Z_p)$ tel que $\kappa\mapsto \phi(x\kappa)$
%soit analytique sur $K_\phi$ pour tout $x\in \GG(\A)$):
%par d\'efinition, on a ${\rm LA}_u(\GG(\A))=\varinjlim_K{\rm LA}(\GG(\A)/K)$,
%o\`u $K$ d\'ecrit les sous-groupes ouverts de $\GG(\wZ)$ de la forme $K^{]p[}K_p$,
%et ${\rm An}(\GG(\A)/K)$ d\'esigne l'espace des $\phi$ telles que,
%pour tout $x\in\GG(\A)$, on ait $\phi(x\kappa)=\phi_{x}(\kappa_p)$, si $\kappa=(\kappa^{]p[},\kappa_p)\in K$,
%o\`u $\phi_x$ est une fonction analytique sur $K_p$.

Si $\theta$ est un caract\`ere du centre $Z(U({\goth g}))$ de l'alg\`ebre enveloppante universelle
de l'alg\`ebre de Lie de $\GG(\Q_p)$, on peut consid\'erer
l'espace
${\rm LA}^\theta(\GG(\A))$, sous-espace de
${\rm LA}(\GG(\A))$ sur lesquel $Z(U({\goth g}))$
agit par le caract\`ere infinit\'esimal $\theta$.

Si $W\in{\rm Irr}(\GG)$, et si $\theta_W$ est le caract\`ere 
\`a travers lequel $Z(U({\goth g}))$ agit sur $W^\dual$, alors
$${\rm LC}(\GG(\A))\otimes W\otimes W^\dual\subset {\rm LA}^{\theta_W}(\GG(\A))$$

\subsubsection{Autres espaces}
On peut varier les conditions de r\'egularit\'e en $p$ \`a loisir, 
par exemple regarder les fonctions de classe ${\cal C}_u^r$, pour $r\in\R_+\sqcup\{\infty\}$.

On peut aussi consid\'erer les fonctions \`a support compact ou tendant vers~$0$ \`a l'infini,
ainsi que les duaux de tous les espaces pr\'ec\'edents, comme par exemple {\it l'espace
${\rm Mes}(\GG(\A),\Lambda)$ des mesures} \`a valeurs dans $\Lambda$, qui est le $\Lambda$-dual
de ${\cal C}_c(\GG(\A),\Lambda)$ (fonctions continues \`a support compact).

\subsubsection{Cohomologie}
On va \'etudier les $H^i(\GG(\Q),X(\GG(\A)))$ pour $X$ une classe de fonctions.
Cela am\`ene naturellement \`a regarder aussi des espaces du type
$H^i(\HH(\Q),X(\GG(\A)))$ o\`u $\HH$ est un sous-groupe alg\'ebrique de $\GG$,
ou encore des $H^i(\HH(\Q),X(\GG(\A)/K))$ o\`u $K$ est un sous-groupe de $\GG(\A)$.
Dans le cas de $H^i(\GG(\Q),X(\GG(\A)))$, on se permet parfois d'abr\'eger
la notation en
$$H^i(\GG,X):=H^i(\GG(\Q),X(\GG(\A)))$$
comme dans l'introduction.

\section{Le borel de ${\bf GL}_2$}
Le but de ce chapitre est de prouver la prop.~\ref{indu5} et le th.\,\ref{indu6}
qui nous servirons \`a comparer les cohomologie et cohomologie \`a support compact
dans le cas de ${\bf GL}_2$.
\Subsection{Le groupe ${\bf G}_a$}

\begin{lemm}\phantomsection\label{gagm1}
Si $\GG={\bf G}_a$ et si $X$ est un des espaces ${\cal C},{\cal C}^{(p)},{\rm LC},{\rm LP},{\rm LA}$,
on a 
$$X(\GG(\A))=X(\GG(\Ai))={\rm Ind}_{\GG(\Z)}^{\GG(\Q)}X(\GG(\wZ))$$
\end{lemm}
\begin{proof}
L'\'egalit\'e $X(\GG(\A))=X(\GG(\Ai))$ d\'ecoule de ce que $\GG(\R)=\R$ est connexe;
le reste de l'\'enonc\'e
r\'esulte de ce que $\GG(\Ai)=\GG(\Q)\cdot\GG(\wZ)$, et l'\'ecriture est
unique \`a $(\gamma,\kappa)\mapsto(\gamma\alpha,\alpha^{-1}\kappa)$, 
avec $\kappa\in\GG(\Q)\cap\GG(\wZ)=\GG(\Z)$ ($=\Z$).
\end{proof}

\begin{prop}\phantomsection\label{gagm3}
Si $\GG={\bf G}_a$,
on a les r\'esultats suivants:
\begin{align*}
%H^i(\GG(\Z),{\rm LC}(\GG(\wZ)))&\cong\begin{cases}
%L&{\text{si $i=0$,}}\\ L&{\text{si $i=1$,}}\\ 0&{\text{si $i\geq 2$}}
%\end{cases} \\
H^i(\GG(\Q),{\cal C}(\GG(\A)))&\cong\begin{cases}
L&{\text{si $i=0$,}}\\ 0&{\text{si $i\geq 1$}}
\end{cases}\\ 
H^i(\GG(\Q),{\rm LC}(\GG(\A)))&\cong\begin{cases}
L&{\text{si $i=0,1$,}}\\ 0&{\text{si $i\geq 2$}}
\end{cases} 
\end{align*}
\end{prop}
\begin{proof}
Comme $X(\A)=X(\Ai)$, si $X={\cal C}, {\rm LC}$,
le lemme de Shapiro implique que $H^i(\GG(\Q),X(\GG(\A)))\cong 
H^i(\Z,X(\wZ))$.
La nullit\'e des $H^i$, pour $i\geq 2$, est imm\'ediate et le 
r\'esultat pour les $H^0$
est une cons\'equence de ce que $\Z$ est dense dans $\wZ$ et donc qu'une fonction invariante
par $\Z$ est constante.  Il reste \`a calculer les $H^1$, i.e.~le conoyau
de $\delta$ d\'efini par $\delta\phi(x)=\phi(x)-\phi(x-1)$.

$\bullet$ Pour $X={\cal C}$, on a ${\cal C}(\wZ,L)=L\otimes{\cal C}(\wZ,\O_L)$
par compacit\'e de $\wZ$, et il suffit donc de prouver
que $H^1(\Z,{\cal C}(\wZ,\O_L))=0$. Comme
${\cal C}(\wZ,\O_L)=\varprojlim_k{\cal C}(\wZ,\O_L/{\goth m}_L^k)$, il suffit, par d\'evissage,
de prouver que $H^1(\Z,{\cal C}(\wZ,k_L))=0$.
Maintenant, ${\cal C}(\wZ,k_L)=\varinjlim_N{\cal C}(\Z/N,k_L)$ et, si $\phi$
est p\'eriodique de p\'eriode $N$, alors $\tilde\phi(x)=\sum_{i=0}^x\phi(i)$
est p\'eriodique de p\'eriode $Np$ et v\'erifie $\tilde\phi(x)-\tilde\phi(x-1)=\phi(x)$.
Autrement dit
l'image de $\delta:{\cal C}(\Z/Np,k_L)\to {\cal C}(\Z/Np,k_L)$ contient ${\cal C}(\Z/N,k_L)$,
et donc $\delta$ est surjective sur ${\cal C}(\wZ,k_L)$; i.e.~$H^1(\Z,{\cal C}(\wZ,k_L))=0$,
ce que l'on voulait d\'emontrer.

$\bullet$ Passons \`a $X={\rm LC}$.
Si $i\in\Z/N$, soit
$e_{N,i}={\bf 1}_{i+N\wZ}$.  
Les $e_{N,i}$ forment une base de ${\rm LC}(\Z/N)$ sur $L$,
et on a $\delta e_{N,i}=e_{N,i}-e_{N,i+1}$; il s'ensuit que 
les $\delta e_{N,i}$ forment une base de ${\rm LC}(\Z/N)_0$ (noyau de la mesure de Haar
sur ${\rm LC}(\Z/N)$).
Comme ${\rm LC}(\wZ)$ est la limite
inductive des ${\rm LC}(\Z/N)$, on obtient une suite exacte
$$0\to {\rm LC}(\wZ)_0\to {\rm LC}(\wZ)\to H^1(\Z,{\rm LC}(\wZ))\to 0$$
Le r\'esultat s'en d\'eduit.
\end{proof}
\begin{rema}\phantomsection\label{gagm4}
$H^1(\GG(\Q),{\rm LC}(\GG(\A)))$ 
admet comme base naturelle l'image du 1-cocycle $r\mapsto r\,{\bf 1}_\A$,
o\`u ${\bf 1}_\A$ est la fonction constante $x\mapsto 1$.
\end{rema}

\begin{rema}\phantomsection\label{gagm2}
La mesure de Haar induit une suite exacte
$$0\to {\rm LC}(\wZ,\O_L)_0\to {\rm LC}(\wZ,\O_L)\to L\to 0$$
qui montre qu'un quotient de deux $\O_L$-modules sans sous-$L$-droite peut
parfaitement \^etre un $L$-module non nul.
On en d\'eduit que
 $$H^1(\GG(\Q),{\rm LC}(\GG(\A),\O_L))\cong L$$

Par contre, si $M_1,M_2$ sont s\'epar\'es et complets pour la topologie $p$-adique, le quotient
ne contient pas de sous-$L$-module non nul (si $M_1/M_2$ contient une suite d'\'el\'ements $v_n$
v\'erifiant $v_n=pv_{n+1}$, et si $\hat v_n$ est un rel\`evement de $v_n$ dans $M_1$, alors
$x_n:=\hat v_n-p\hat v_{n+1}\in M_2$ et donc $x_0+px_1+p^2x_2+\cdots$ converge dans $M_2$
et donc aussi dans $M_1$ vers la m\^eme limite (car $M_1$ est s\'epar\'e),
et comme cette limite est $\hat v_0$, cela implique $v_0=0$).
\end{rema}

%\begin{prop}\label{gagm3.1}
%Si $\GG={\bf G}_a$, on a 
%$$H^i(\GG,{\rm LP})=\begin{cases} L&{\text{si $i=0$}}\\0&{\text{si $i\geq1$}}\end{cases}$$
%\end{prop}
%\begin{proof}
%Comme dans la preuve de la prop.\,\ref{gagm3}, on est ramen\'e
%\`a calculer $H^i(\Z,{\rm LP}(\wZ,L))$, et le r\'esultat est imm\'ediat pour $i\neq 1$.
%La nullit\'e de $H^1(\Z,{\rm LP}(\wZ,L))$ r\'esulte de 
%celle de $H^1(N\Z,{\rm LP}(\Z/N,L))$ pour tout $N$: la suite d'inflation-restriction
%montre que cela implique $H^1(\Z,{\rm LP}(\Z/N,L))=0$ et ${\rm LP}(\wZ,L)$ est la limite
%inductive des ${\rm LP}(\Z/N,L)$. La nullit\'e de $H^1(N\Z,{\rm LP}(\Z/N,L))$ re\'sulte
%de ce que ${\rm LP}(\Z/N,L)$ est, en tant que $N\Z$-module,
% une somme finie de copies de l'espace des polyn\^omes
%et que tout polyn\^ome est de la forme $\phi(x)-\phi(x-N)$ 
%avec $\phi$ un polyn\^ome (de degr\'e plus
%grand).
%\end{proof}
\Subsection{Le groupe ${\bf G}_m$}
On suppose maintenant que $\GG={\bf G}_m$.
Si $\Lambda$ est un sous-anneau de $\Q$,
on note $\GG(\Lambda)_+$ l'ensemble des $x\in\GG(\Lambda)$ tels que $x_\infty\in\GG(\R)_+$
(en particulier, $\GG(\Z)_+=\{1\}$).
\begin{lemm}\phantomsection\label{gagm5}
Si $\GG={\bf G}_m$ et si $X$ est un des espaces ${\cal C},{\cal C}^{(p)},{\rm LC},{\rm LP},{\rm LA}$,
on a 
$$X(\GG(\Ai))\cong {\rm Ind}_{\GG(\Z)}^{\GG(\Q)}X(\GG(\wZ))
\quad{\rm et}\quad
X(\GG(\A))\cong {\rm Ind}_{\GG(\Z)_+}^{\GG(\Q)}X(\GG(\wZ))$$
\end{lemm}
\begin{proof}
Pour $\GG(\Ai)$, cela r\'esulte de ce que $\GG(\Ai)=\GG(\Q)\cdot\GG(\wZ)$, et l'\'ecriture est
unique \`a $(\gamma,\kappa)\mapsto(\gamma\alpha,\alpha^{-1}\kappa)$ pr\`es, avec $\alpha\in\GG(\Z)$.
Pour $\GG(\A)$, cela r\'esulte de ce que $X(\GG(\wZ))=X(\GG(\wZ)\times\GG(\R)_+)$
et de ce que $\GG(\A)=\GG(\Q)\cdot(\GG(\wZ)\times\GG(\R)_+)$, et l'\'ecriture est
unique \`a $(\gamma,\kappa, g_\infty)\mapsto(\gamma\alpha,\alpha^{-1}\kappa,\alpha^{-1}g_\infty)$
pr\`es, avec $\alpha\in\GG(\Z)_+$ (et donc est unique).
\end{proof}

\begin{prop}\phantomsection\label{gagm6}
Si $\GG={\bf G}_m$, et si $X={\cal C},{\rm LC}$, alors 
\begin{align*}
H^i(\GG(\Q),X(\GG(\A)))&\cong\begin{cases}
X(\GG(\wZ)) &{\text{si $i=0$,}}\\
0 &{\text{si $i\geq 1$.}}
\end{cases}
%H^i(\GG(\Q),{\rm LC}(\GG(\A)))&\cong\begin{cases}
%{\rm LC}(\GG(\wZ),\O_L) &{\text{si $i=0$,}}\\
%0 &{\text{si $i\geq 1$.}}
%\end{cases}
\end{align*}
l'action de
$\GG(\A)=\A^\dual$ sur les $H^0$ se faisant \`a travers $\A^\dual/\R_+^\dual\Q^\dual\cong\wZ^\dual=\GG(\wZ)$.
\end{prop}
\begin{proof}
Cela r\'esulte du lemme de Shapiro: on a $\GG(\Z)_+=\{1\}$.
\end{proof}

Si $\eta:\GG(\Q)\to L^\dual$ est un caract\`ere, notons 
$$\eta_\A:\GG(\A)\to L^\dual$$
le caract\`ere d\'efini par $\eta_\A(r\,a\,x_\infty)=\eta(r)$, si $r\in\Q^\dual$, $a\in\wZ^\dual$
et $x_\infty\in\R_+^\dual$.  Par exemple, si $\eta(r)=r^{-1}$, alors $\eta_\A$
est le caract\`ere $\delta_\A$ du \no\ref{prelim4}.

Le r\'esultat suivant est imm\'ediat.
\begin{lemm}\phantomsection\label{gagm6.1}
L'application $\phi\otimes\eta_\A\mapsto\eta_\A\phi$ induit un isomorphisme
$\GG(\A)$-\'equivariant
$$H^0(\GG(\Q),X)\otimes\eta_\A\overset{\sim}{\to} H^0(\GG(\Q),X\otimes\eta)$$
\end{lemm}

\Subsection{L'unipotent et le l\'evi}
Dans le reste de ce chapitre, $\GG={\bf GL}_2$.
Soit $\BB=\matrice{*}{*}{0}{*}$ le borel de $\GG$.
On a $\BB=\UU\LL$, o\`u $\UU=\matrice{1}{*}{0}{1}$ est l'unipotent et $\LL=\matrice{*}{0}{0}{*}$
est le l\'evi de $\BB$ (un tore d\'eploy\'e).
On a 
\begin{equation}\phantomsection\label{scindage1}
\BB(\Ai)\backslash\GG(\Ai)\cong\BB(\wZ)\backslash\GG(\wZ)\cong\piqp(\wZ)
\end{equation}
\begin{lemm}\phantomsection\label{scindage}
La projection naturelle $\GG(\wZ)\to\piqp(\wZ)$ admet
un scindage continu {\rm (et m\^eme localement analytique)}.
\end{lemm}
\begin{proof}
La fl\`eche naturelle
$\BB(\wZ)\backslash\GG(\wZ)\to \piqp(\wZ)$ envoie $\matrice{x}{y}{z}{t}$ sur
$\frac{z}{t}$ et on utilise comme scindage le produit des scindages
$\piqp(\Z_\ell)\to \GG(\Z_\ell)$ envoyant
$r\in\piqp(\Z_\ell)=\piqp(\Q_\ell)$ sur $\matrice{1}{0}{r}{1}$
si $r\in\Z_\ell$, et sur $\matrice{0}{-1}{1}{1/r}$ si $r\in\piqp(\Q_\ell)\moins\Z_\ell$.
\end{proof}

\begin{lemm}\phantomsection\label{scindage2}
Soit $I_\infty=\matrice{-1}{0}{0}{-1}\in\GG(\R)$.

{\rm (i)} On a une factorisation
$\BB(\A)$-\'equivariante
$${\cal C}(\GG(\A))\cong{\cal C}(\piqp(\wZ))\wotimes{\cal C}(\BB(\A))^{I_\infty=1}$$
o\`u $\BB(\A)$ agit trivialement sur ${\cal C}(\piqp(\wZ))$.

{\rm (ii)} On a une factorisation
$\LL(\A)$-\'equivariante
$${\rm LC}(\UU(\A)\backslash\GG(\A))\cong{\rm LC}(\piqp(\wZ))\wotimes{\rm LC}(\LL(\A))^{I_\infty=1}$$
o\`u $\LL(\A)$ agit trivialement sur ${\rm LC}(\piqp(\wZ))$.
\end{lemm}
\begin{proof}
Cela r\'esulte des d\'ecompositions
$\GG(\A)=\BB(\Ai)\times\piqp(\wZ)\times\GG(\R)$ et
$\UU(\A)\backslash\GG(\A)=\BB(\Ai)\times\piqp(\wZ)\times(\UU(\R)\backslash\GG(\R))$.
La raison pour laquelle il faut prendre les points fixes par $I_\infty$ est que $1$
et $I_\infty$ sont dans la m\^eme composante connexe de $\GG(\R)$ mais pas dans la
m\^eme composante connexe de $\BB(\R)$ ou $\LL(\R)$. 
\end{proof}

\begin{lemm}\phantomsection\label{gagm8}
Si $k\in\N$ et $j\in\Z$, on a des isomorphismes de
$\LL(\Q)$-modules
\begin{align*}
H^i(\UU(\Q),{\rm LC}(\UU(\A))\otimes W_{k,j}) \cong\begin{cases}
L(a^{k+j}\otimes d^j)&{\text{si $i=0$,}}\\
L(a^{j-1}\otimes d^{k+j+1})&{\text{si $i=1$,}}\\
0&{\text{si $i\geq 2$.}}\end{cases}
\end{align*}
\end{lemm}
\begin{proof}
On a
$$\matrice{a}{b}{0}{d}\cdot(e_1^ie_2^{k-i}(e_1\wedge e_2)^j)=
(ad)^j(ae_1)^i(be_1+de_2)^{k-i}(e_1\wedge e_2)^j$$
On en d\'eduit une suite exacte de repr\'esentations de $\BB(\Q)$
$$0\to e_1\otimes W_{k-1,j}\to W_{k,j}\to L(a^j\otimes d^{k-j})\to 0$$
(le caract\`ere $a^j\otimes d^{k-j}$ est celui par lequel $\BB(\Q)$ agit
sur $e_2^k$ modulo $e_1\otimes W_{k-1,j}$).
On en d\'eduit, par r\'ecurrence sur $k$
que $H^0(\UU(\Q),{\rm LC}(\UU(\A))\otimes W_{k,j})=L$
(engendr\'e par ${\bf 1}_{\UU(\A)}e_1^k$)
et que $H^1(\UU(\Q),{\rm LC}(\UU(\A))\otimes W_{k,j})=L$
(engendr\'e par la classe du cocycle
$\matrice{1}{b}{0}{1}\mapsto {\bf 1}_{\UU(\A)}\frac{(be_1+de_2)^{k+1}-e_2^{k+1}}{(k+1)e_1}$,
i.e. $b{\bf 1}_{\UU(\A)}e_2^k$ modulo $e_1\otimes W_{k-1,j}$; notons que le cocycle 
correspondant pour $W_{k-1,j}$
(multipli\'e par $e_1$) meurt
dans $H^1(\UU(\Q),{\rm LC}(\UU(\A))\otimes W_{k,j})$ 
(c'est le bord de $\frac{1}{k}e_2^{k}$), ce qui fait
que $H^1(\UU(\Q),{\rm LC}(\UU(\A))\otimes W_{k,j})\to H^1(\UU(\Q),{\rm LC}(\UU(\A)))$
est un isomorphisme, et permet de faire fonctionner la r\'ecurrence).

L'action de $\LL(\A)$ sur $H^0(\UU(\Q),{\rm LC}(\UU(\A))\otimes W_{k,j})$ est
celle sur $e_1^k(e_1\wedge e_2)^j$, i.e.~$a^{j+k}\otimes d^j$, ce qui fournit
le r\'esultat pour $H^0$.  

Pour $H^1$, il y a une torsion suppl\'ementaire par
rapport \`a l'action sur $e_2^k(e_1\wedge e_2)^j$: en effet,
$\gamma\in\LL(\Q)$ agit sur un $1$-cocycle
$u\mapsto\phi_u$ en l'envoyant sur le $1$-cocycle 
$u\mapsto \gamma*\phi_{\gamma^{-1}u\gamma}$, et comme
$\matrice{a^{-1}}{0}{0}{d^{-1}}\matrice{1}{r}{0}{1}\matrice{a}{0}{0}{d}=\matrice{1}{a^{-1}dr}{0}{1}$,
l'action de $\matrice{a}{0}{0}{d}$ est multipli\'ee
par $a^{-1}d$ par rapport \`a celle sur $e_2^k(e_1\wedge e_2)^j$;
on obtient donc le caract\`ere $a^{j-1}\otimes d^{k+j+1}$.
\end{proof}

\begin{lemm}\phantomsection\label{indu61}
Si $\eta:\LL(\Q)\to L^\dual$ est un caract\`ere, alors
$$H^i(\LL(\Q),{\rm LC}(\UU(\A)\backslash\GG(\A))\otimes\eta)\cong
\begin{cases}
{\rm Ind}_{\BB(\A)}^{\GG(\A)} \big(H^0(\LL,{\rm LC})\otimes\eta_\A\big)&{\text{si $i=0$,}}\\
0&{\text{si $i\geq 1$.}}\end{cases}$$
\end{lemm}
\begin{proof}
La factorisation du (ii) du lemme~\ref{scindage2} alli\'ee au lemme~\ref{gagm6}
fournit le r\'esultat pour $i\geq 1$.

Pour $i=0$,
le membre de gauche s'identifie \`a l'espace des $\phi:\GG(\A)\times\LL(\A)\to L$
v\'erifiant $\phi(x,y\ell)=\phi(x,y)$ si $\ell\in\LL(\A)$ (et donc $\phi(x,y)=\phi(x,1)$
et la seconde variable ne joue pas vraiment de r\^ole), $\phi(u^{-1}x,y)=\phi(x,y)$
si $u\in\UU(\A)$, et $\eta(\lambda)\phi(\lambda^{-1}x,y)=\phi(x,y)$ si $\lambda\in\LL(\Q)$;
l'action de $\GG(\A)$ \'etant $(g\star\phi)(x,y)=\phi(xg,y)$.

Posons $\phi^\dual(z,t)=\phi(tz,t)$ (et donc $\phi(x,y)=\phi^\dual(y^{-1}x,y)$).
Les conditions ci-dessus deviennent $\phi^\dual(\ell^{-1}z,t\ell)=\phi^\dual(z,t)$ si $\ell\in\LL(\A)$,
$\phi^\dual(t^{-1}utz,t)=\phi^\dual(z,t)$ si $u\in\UU(\A)$ (comme $u\mapsto t^{-1}ut$ est
un isomorphisme de $\UU(\A)$, cette condition \'equivaut \`a 
$\phi^\dual(u^{-1}z,t)=\phi^\dual(z,t)$ si $u\in\UU(\A)$), et
$\eta(\lambda)\phi^\dual(\lambda^{-1}z,t)=\phi^\dual(z,t)$ si $\lambda\in\LL(\Q)$
(compte-tenu de la premi\`ere condition, et de la commutativit\'e de $\LL$,
cette condition \'equivaut \`a $\eta(\lambda)\phi^\dual(z,\lambda^{-1}t)=\phi^\dual(z,t)$ 
si $\lambda\in\LL(\Q)$);
l'action de $\GG(\A)$ restant $(g\star\phi^\dual)(z,t)=\phi^\dual(zg,y)$.

La derni\`ere condition se traduit par $\phi^\dual\in 
{\rm LC}(\GG(\A))\otimes H^0(\LL,{\rm LC}\otimes\eta)$, et on a
$H^0(\LL,{\rm LC}\otimes\eta)=\eta_\A H^0(\LL,{\rm LC})$, 
cf.~lemme~\ref{gagm6.1}.
Les deux premi\`eres conditions se traduisent par l'invariance par $\BB(\A)$, pour
l'action $(b*\phi^\dual)(z,t)=\phi(b^{-1}z,t\overline b)$ o\`u $b\mapsto \overline b$
est la projection naturelle $\BB\to\LL$.  On reconnait la d\'efinition
de l'induite, ce qui permet de conclure. 
\end{proof}

\Subsection{Le borel}

\Subsubsection{Cohomologie \`a valeurs dans des espaces fonctionnels sur $\BB(\A)$>}
\begin{prop}\phantomsection\label{gagm7}
On a:
\begin{align*}
H^i(\BB(\Q),{\cal C}(\BB(\A)))&\cong \begin{cases}
{\cal C}(\LL(\wZ)) &{\text{si $i=0$}}\\
0&{\text{si $i\geq 1$.}}\end{cases}\\
H^i(\BB(\Q),{\rm LC}(\BB(\A)))&\cong\begin{cases}
{\rm LC}(\LL(\wZ)) &{\text{si $i=0$}}\\
{\rm LC}(\LL(\wZ))\otimes(\delta_\A\otimes\delta_\A^{-1}) &{\text{si $i=1$}}\\
0&{\text{si $i\geq 2$.}}\end{cases}
\end{align*}
l'action de
$\BB(\A)$ se faisant \`a travers $\LL(\A)$ qui agit sur
${\cal C}(\LL(\wZ))$ et ${\rm LC}(\LL(\wZ))$
\`a travers $\LL(\A)/\LL(\R)_+\LL(\Q)\cong\LL(\wZ)$.
\end{prop}
\begin{proof}
On utilise la suite spectrale de Hochschild-Serre pour 
$$1\mapsto \UU(\Q)\to\BB(\Q)\to\LL(\Q)\to 1$$
La d\'ecomposition $\BB=\UU\LL$ fournit des factorisations
\begin{align*}
{\rm LC}(\BB(\A))&\cong{\rm LC}(\UU(\A))\otimes{\rm LC}(\LL(\A))\\
{\cal C}(\BB(\A))&\cong{\cal C}(\UU(\A))\wotimes{\cal C}(\LL(\A))
\end{align*}
Compte-tenu des~prop.\,\ref{gagm3} et~\ref{gagm6}, les seuls $H^i(\LL(\Q),H^j(\UU(\Q),-))$ non nuls
sont $i=j=0$ pour ${\cal C}$ et ${\rm LC}$, et $i=0,j=1$ pour ${\rm LC}$.

On en d\'eduit le r\'esultat pour les $H^0$ via la prop.\,\ref{gagm6}, 
puisque les $H^0(\UU(\Q),X(\UU(\A)))$ sont \'egaux \`a $L$,
avec action triviale de $\LL(\Q)$ (prop.\,\ref{gagm3}).  Par contre, l'action de $\LL(\Q)$ sur
$H^1(\UU(\Q),{\rm LC}(\UU(\A)))\cong L$ n'est pas triviale: $\gamma$ agit sur un $1$-cocycle
$u\mapsto\phi_u$ en l'envoyant sur le $1$-cocycle $u\mapsto \gamma*\phi_{\gamma^{-1}u\gamma}$, et comme
$\matrice{a^{-1}}{0}{0}{d^{-1}}\matrice{1}{r}{0}{1}\matrice{a}{0}{0}{d}=\matrice{1}{a^{-1}dr}{0}{1}$,
l'action de $\matrice{a}{0}{0}{d}$ est la multiplication par $a^{-1}d$ puisque $\gamma*{\bf 1}_\A={\bf 1}_\A$.
Notons $a^{-1}\otimes d$ 
le caract\`ere $\matrice{a}{0}{0}{d}\mapsto a^{-1}d$. Il r\'esulte de ce qui pr\'ec\`ede
que 
$$H^1(\BB(\Q),{\rm LC}(\BB(\A)))\cong H^0(\LL(\Q),{\rm LC}(\LL(\A))
\otimes L(a^{-1}\otimes d))$$
et le r\'esultat est une cons\'equence du lemme~\ref{gagm6.1} et du fait que
le caract\`ere de $\LL(\A)$ associ\'e \`a $a^{-1}\otimes d$ est
$\delta_\A\otimes\delta_\A^{-1}$.
\end{proof}

\begin{rema}\phantomsection\label{gagm2.1}
La m\^eme preuve montre, en utilisant la rem.\,\ref{gagm2}, que
$$H^1(\BB(\Q),{\rm LC}(\BB(\A),\O_L))\cong
{\rm LC}(\LL(\wZ),L)\otimes(\delta_\A\otimes\delta_\A^{-1})$$
\end{rema}

\Subsubsection{Induction de $\BB$ \`a $\GG$}\label{indu0}
\begin{lemm}\phantomsection\label{indu1}
Si $r\leq s$,
${\rm Ind}_\BB^\GG(a^r\otimes d^s)={\rm Sym}^{s-r}\otimes\det^r$.
\end{lemm}
\begin{proof}
On a $\matrice{a}{b}{0}{d}\matrice{x}{y}{z}{t}=\matrice{ax+bz}{ay+bt}{dz}{dt}$.
Il s'ensuit que ${\rm Ind}_\BB^\GG(a^r\otimes d^s)$ est l'espace
des $\phi\in L[x,y,z,t,(xt-yz)^{-1}]$ v\'erifiant
$P({ax+bz},{ay+bt},{dz},{dt})=a^rd^sP(x,y,z,t)$ pour tout $\matrice{a}{b}{0}{d}\in\BB$.
On peut \'ecrire $P=(xy-zt)^r\tilde P$ et on a
$\tilde P({ax+bz},{ay+bt},{dz},{dt})=d^{s-r}\tilde P(x,y,z,t)$ pour tout $\matrice{a}{b}{0}{d}$.
Utilisant cette identit\'e pour $\matrice{a}{0}{0}{1}$ prouve que
$\tilde P$ est constant en $x$ et $y$, et pour $\matrice{1}{0}{0}{d}$, que $\tilde P$
est homog\`ene de degr\'e total $s-r$ en $z$ et $t$.  Le r\'esultat s'en d\'eduit.
\end{proof}

\begin{lemm}\phantomsection\label{indu2}
Si $V$ est une repr\'esentation lisse de $\BB(\A)$
et si $r\leq s\in\Z$,
alors\footnote{L'induite du membre de gauche est l'induite localement alg\'ebrique;
celle du membre de droite est l'induite lisse.}
$${\rm Ind}_{\BB(\A)}^{\GG(\A)}\big(V\otimes(a_p^r\otimes d_p^s)\big)=
\big({\rm Ind}_{\BB(\A)}^{\GG(\A)} V\big)
\otimes {\rm Sym}^{s-r}_p\otimes{\det}^r_p$$ 
o\`u $\GG(\A)$ agit sur
${\rm Sym}^{s-r}_p$ et $\det^r_p$ \`a travers $\GG(\Q_p)$, et ${\rm Sym}^{s-r}_p$ et $\det^r_p$
sont les repr\'esentations alg\'ebriques usuelles de $\GG(\Q_p)$, et
$a_p^r\otimes d_p^s$ est le caract\`ere $\matrice{a}{b}{0}{d}\mapsto a_p^rd_p^s$ de $\B(\A)$.
\end{lemm}
\begin{proof}
On a ${\rm LP}(\GG(\A))={\rm LC}(\GG(\A))\otimes {\rm Alg}(\GG)$.
Le membre de droite est l'espace des points fixes
de $\big({\rm LC}(\GG(\A))\otimes V\big)
\otimes\big({\rm Alg}(\GG)\otimes (a^r\otimes b^s)\big)$
sous l'action de $\BB(\A)\times\BB$ (o\`u $\BB$ est consid\'er\'e comme un groupe
alg\'ebrique et muni de la topologie de Zariski),
et le membre de gauche est celui des points fixes sous l'action de $\BB(\A)$
s'envoyant dans $\BB(\A)\times\BB$ par l'application ${\rm id}\times (x\mapsto x_p)$.
Ces deux espaces co\"{\i}ncident car l'image de $\BB(\A)$ est dense dans 
$\BB(\A)\times\BB$.
\end{proof}
\begin{rema}\phantomsection\label{indu1.1}
Si $r> s$, alors
$${\rm Ind}_\BB^\GG(a^r\otimes d^s)=0
\quad{\rm et}\quad
{\rm Ind}_{\BB(\A)}^{\GG(\A)}\big(V\otimes(a_p^r\otimes d_p^s)\big)=0$$ 
En effet, en \'ecrivant $P=(xy-zt)^s\tilde P$,
on obtient $\tilde P({ax+bz},{ay+bt},{dz},{dt})=a^{r-s}\tilde P(x,y,z,t)$ 
pour tout $\matrice{a}{b}{0}{d}$, ce qui implique que $\tilde P$ est constant
en $x,y$ et homog\`ene de degr\'e~$r-s$ en $x,y$.
\end{rema}

\begin{rema}\phantomsection\label{indu3}
Si $M$ est une $L$-repr\'esentation continue de $\BB(\A)$, 
alors 
$${\rm Ind}_{\BB(\A)}^{\GG(\A)}M={\rm Ind}_{\BB(\A)}^{\GG(\A)}M^{I_\infty=1}$$
En effet, $M=M^{I_\infty=1}\oplus M^{I_\infty=-1}$, et comme $I_\infty\in\GG(\R_+)$,
on a $\phi((I_\infty)x)=\phi(x)$ si $\phi:\GG(\A)\to M$ est continue. Par ailleurs,
comme $I_\infty\in\B(\A)$, on a $\phi((I_\infty)x)=\pm\phi(x)$ si $\phi\in
{\rm Ind}_{\BB(\A)}^{\GG(\A)}M^{I_\infty=\pm}$.  Il en r\'esulte
que ${\rm Ind}_{\BB(\A)}^{\GG(\A)}M^{I_\infty=-1}=0$.

De m\^eme, comme $\BB(\wZ)\backslash\GG(\wZ)\cong \BB(\Ai)\backslash\GG(\Ai)$,
$${\rm Res}_{\GG(\A)}^{\GG(\wZ)}\,{\rm Ind}_{\BB(\A)}^{\GG(\A)}M\cong
{\rm Ind}_{\BB(\wZ)}^{\GG(\wZ)}{\rm Res}_{\BB(\A)}^{\BB(\wZ)}M^{I_\infty}$$
\end{rema}

\Subsubsection{Cohomologie de $\BB(\Q)$ \`a valeurs dans des induites}

\begin{prop}\phantomsection\label{indu5}
On a des isomorphismes de repr\'esentations de $\GG(\A)$
$$H^i(\BB(\Q),{\cal C}(\GG(\A)))\cong\begin{cases}
{\rm Ind}_{\BB(\A)}^{\GG(\A)}{\cal C}(\LL(\wZ))&{\text{si $i=0$,}}\\
0 &{\text{si $i\geq 1$.}}\end{cases}$$
l'action de $\BB(\A)$ sur ${\cal C}(\LL(\wZ))$ se faisant \`a travers
le quotient $\LL(\wZ)\cong \LL(\A)/\LL(\R)_+\LL(\Q)$ de $\BB(\A)$.
\end{prop}
\begin{proof}
D'apr\`es le lemme~\ref{scindage2}, on a une factorisation
$\BB(\A)$-\'equivariante 
${\cal C}(\GG(\A))\cong{\cal C}(\piqp(\wZ))\wotimes{\cal C}(\BB(\A))^{I_\infty=1}$,
o\`u $\BB(\A)$ agit trivialement sur ${\cal C}(\piqp(\wZ))$.
Comme 
${\cal C}(\BB(\A))={\cal C}(\BB(\A))^{I_\infty=1}\oplus {\cal C}(\BB(\A))^{I_\infty=-1}$, et comme
$H^i(\BB,{\cal C})=0$ si $i\geq 1$ (prop.\,\ref{gagm7}),
cela prouve que $H^i(\BB(\Q),{\cal C}(\GG(\A)))=0$,
pour $i\geq 1$.

Passons au calcul de $H^0(\BB(\Q),{\cal C}(\GG(\A)))$. On a
$${\cal C}(\GG(\A))={\rm Ind}_{\BB(\A)}^{\GG(\A)}{\cal C}(\BB(\A))
=H^0(\BB(\A),{\cal C}(\GG(\A)\times\BB(\A)))$$
o\`u $b\in\BB(\A)$ agit (\`a droite) sur $\GG(\A)\times\BB(\A)$ par $b\cdot(x,y)=(b^{-1}x,yb)$.
Si $\beta\in\BB(\Q)$ agit (\`a droite) sur $\GG(\A)\times\BB(\A)$ 
par $\beta\cdot(x,y)=(x,\beta^{-1}y)$, on a:
\begin{align*}
H^0(\BB(\Q),{\cal C}(\GG(\A)))&=H^0(\BB(\Q),H^0(\BB(\A),{\cal C}(\GG(\A)\times\BB(\A))))\\
&=H^0(\BB(\A),H^0(\BB(\Q),{\cal C}(\GG(\A)\times\BB(\A))))\\
&={\rm Ind}_{\BB(\A)}^{\GG(\A)}H^0(\BB(\Q),{\cal C}(\BB(\A)))
\end{align*}
On conclut en utilisant la prop.\,\ref{gagm7}.
\end{proof}

\begin{prop}\phantomsection\label{indu6.0}
Si $k\in\N$ et $j\in\Z$, alors\footnote{$W_{k,j}={\rm Sym}^k\otimes\det^j$,
$\delta_\A$ et $|\ |_{\A,p}$ sont les caract\`eres de $\A^\dual$ d\'efinis
au \no\ref{prelim4},  en particulier $|\ |_{\A,p}$ 
est unitaire contrairement \`a $\delta_\A$,
et on a $|x|_{\A,p}=x_p\delta_\A(x)$.}
\begin{align*}
H^i(\BB(\Q),{\rm LC}(\GG(\A))\otimes (W_{k,j}\otimes W_{k,j}^\dual)) \cong
\begin{cases}
{\rm Ind}_{\BB(\A)}^{\GG(\A)}{\rm LC}(\LL(\wZ))_{k,j}&\text{si $i=0$,}\\
{\rm Ind}_{\BB(\A)}^{\GG(\A)}{\rm LC}(\LL(\wZ))'_{k,j}&\text{si $i=1$,}\\
0&\text{si $i\geq 2$.}\end{cases}
\end{align*}
O\`u ${\rm LC}(\LL(\wZ))_{k,j}$ et ${\rm LC}(\LL(\wZ))'_{k,j}$ sont les twists:
\begin{align*}
{\rm LC}(\LL(\wZ))_{k,j}&:={\rm LC}(\LL(\wZ))\otimes (|a|_{\A,p}^{-k-j}\otimes
|d|_{\A,p}^{-j})\\
{\rm LC}(\LL(\wZ))'_{k,j}&:={\rm LC}(\LL(\wZ))\otimes (|a|_{\A,p}^{-k-j}\otimes
|d|_{\A,p}^{-j})\otimes(\delta_\A\otimes\delta_\A^{-1})^{k+1}
\end{align*}
\end{prop}
\begin{proof}
On utilise la suite spectrale de Hochschild-Serre associ\'ee
au d\'evissage $1\to\UU(\Q)\to\BB(\Q)\to\LL(\Q)\to 1$.  Par ailleurs,
$\BB(\Q)$ agit trivialement sur $W_{k,j}^\dual$ qui peut donc \^etre sorti
de la cohomologie.
D'apr\`es le lemme~\ref{gagm8}, $H^i(\UU(\Q),{\rm LC}(\GG(\A))\otimes W_{k,j})$
est $0$ pour $i\geq 2$ et de la forme ${\rm LC}(\UU(\A)\backslash\GG(\A))\otimes\eta_{k,j}^{(i)}$,
si $i=0,1$, o\`u $\eta_{k,j}^{(0)}=a^{k+j}\otimes d^j$ et
$\eta_{k,j}^{(1)}=a^{j-1}\otimes d^{k+j+1}$.  On peut donc utiliser le lemme~\ref{indu61}
pour en d\'eduire que $H^i(\BB(\Q),-)=H^0(\LL(\Q),H^i(\UU(\Q),-))$ (tous les autres $E_2$-termes
de la suite spectrale sont $0$),
et obtenir la formule (pour $i=0,1$)
$$H^i(\BB(\Q),{\rm LC}(\GG(\A))\otimes (W_{k,j}\otimes W_{k,j}^\dual))\cong
\big({\rm Ind}_{\BB(\A)}^{\GG(\A)}\big(H^0(\LL,{\rm LC})\otimes\eta_{k,j,\A}^{(i)}\big)\big)\otimes
W_{k,j}^\dual$$

Pour conclure, on \'ecrit $W_{k,j}^\dual=W_{k,-k-j}$ sous la forme
${\rm Ind}_{\BB(\Q_p)}^{\GG(\Q_p)}(a_p^{-k-j}\otimes d_p^{-j})$ et on utilise
la prop.\,\ref{gagm6} pour mettre $H^0(\LL,{\rm LC})$ sous la forme ${\rm LC}(\LL(\wZ))$,
le lemme~\ref{indu2} pour faire rentre $W_{k,-k-j}$ dans l'induite,
et les identit\'es
\begin{align*}
(a^{j+k}{\hskip.5mm\otimes\hskip.5mm}{d^j})\otimes(a_p^{-k-j}{\hskip.5mm\otimes\hskip.5mm} d_p^{-j})
&=|\ |_{\A,p}^{-k-j}{\hskip.5mm\otimes\hskip.5mm}|\ |_{\A,p}^{-j}\\
\hskip-2cm(a^{j-1}{\hskip.5mm\otimes\hskip.5mm}{d^{k+j+1}})\otimes(a_p^{-k-j}{\hskip.5mm\otimes\hskip.5mm} d_p^{-j})
&=\big(|\ |_{\A,p}^{-k-j}{\hskip.5mm\otimes\hskip.5mm}|\ |_{\A,p}^{-j}\big)
\otimes (\delta_\A{\hskip.5mm\otimes\hskip.5mm}\delta_\A^{-1})^{k+1}
\qedhere
\end{align*}
\end{proof}

Si $\pi$ est une $L$-repr\'esentation localement alg\'ebrique de $\GG(\A)$, on note
$\pi^\vee$ sa contragr\'ediente, i.e.~l'ensemble des vecteurs $\GG(\wZ)$-finis de
$\pi^\dual$. 
\begin{exem}\phantomsection\label{exindu}
(i)
Si $\eta_1,\eta_2:\Aidu\to L^\dual$ sont des caract\`eres lisses, alors
$$\big({\rm Ind}_{\BB(\A)}^{\GG(\A)}(\eta_1\otimes\eta_2)\big)^\vee\cong
{\rm Ind}_{\BB(\A)}^{\GG(\A)}(\eta_1^{-1}\delta_\A\otimes\eta_2^{-1}\delta_\A^{-1})$$
(ii) Si $\eta_1,\eta_2:\Aidu\to L^\dual$ sont des caract\`eres lisses,
et si $s\geq r$, alors
\begin{align*}
\big({\rm Ind}_{\BB(\A)}^{\GG(\A)}(\eta_1a_p^r{\hskip.3mm\otimes\hskip.3mm}\eta_2d_p^s)\big)^\vee&\cong
{\rm Ind}_{\BB(\A)}^{\GG(\A)}
(\eta_1^{-1}\delta_\A a_p^{-s}{\hskip.3mm\otimes\hskip.3mm}\eta_2^{-1}\delta_\A^{-1}d_p^{-r})\\
\big({\rm Ind}_{\BB(\A)}^{\GG(\A)}
(\eta_1\,|\ |_{\A,p}^r{\hskip.3mm\otimes\hskip.3mm}\eta_2\,|\ |_{\A,p}^s)\big)^\vee&\cong
{\rm Ind}_{\BB(\A)}^{\GG(\A)}\big((\eta_1^{-1}|\ |_{\A,p}^{-s}
{\hskip.3mm\otimes\hskip.3mm}\eta_2^{-1}|\ |_{\A,p}^{-r})
{\hskip.3mm\otimes\hskip.3mm}(\delta_\A\otimes\delta_\A^{-1})^{s-r+1}\big)
\end{align*}
Le (i) est classique, le (ii) s'en d\'eduit en utilisant le lemme~\ref{indu2} et 
l'identit\'e 
$$({\rm Sym}^{s-r}\otimes{\det}^r)^\dual\cong{\rm Sym}^{s-r}\otimes{\det}^{-s}$$
Remarquons que ces formules sont encore valables si $s<r$, mais alors les deux membres sont $0$
(cf.~rem.\,\ref{indu1.1}).
\end{exem}
\begin{theo}\phantomsection\label{indu6}
On les isomorphismes suivants de repr\'esentations de $\GG(\A)$:
\begin{align*}
H^i(\BB(\Q),{\rm LP}(\GG(\A))) \cong
\begin{cases}
{\rm Ind}_{\BB(\A)}^{\GG(\A)}{\rm LP}(\LL(\wZ))&\text{si $i=0$,}\\
\big({\rm Ind}_{\BB(\A)}^{\GG(\A)}{\rm LP}(\LL(\wZ))\big)^\vee&\text{si $i=1$,}\\
0&\text{si $i\geq 2$.}\end{cases}
\end{align*}
\end{theo}
\begin{proof}
Le cas $i=0$ r\'esulte de la prop.\,\ref{indu6.0} en prenant la somme directe sur $k\in\N$
et $j\in\Z$, en utilisant la d\'ecomposition
$${\rm LP}(\LL(\wZ))=\bigoplus\nolimits_{r,s\in\Z}{\rm LC}(\LL(\wZ)) 
\big(|\ |_{\A,p}^{r}\otimes |\ |_{\A,p}^{s}\big)$$
 et le
fait que ${\rm Ind}_{\BB(\A)}^{\GG(\A)}\big({\rm LC}(\LL(\wZ)) (|\ |_{\A,p}^{r}
\otimes |\ |_{\A,p}^{s})\big)=0$ si $s<r$.

Pour le cas $i=1$, on \'etend les scalaires \`a $\Qbar_p$ pour \'ecrire
\begin{equation}\phantomsection\label{indu6.7}
{\rm LP}(\LL(\wZ))=\bigoplus  \big(\eta_1|\ |_{\A,p}^{r}\otimes \eta_2|\ |_{\A,p}^{s}\big)
\end{equation}
o\`u
la somme porte sur $r,s\in\Z$ et sur $\eta_1,\eta_2$ caract\`eres lisses de $\wZ^\dual$ 
(vus comme
des caract\`eres de $\Aidu/\R_+^\dual\Q^\dual$).  On utilise alors la prop.\,\ref{indu6.0}
et le (ii)
des exemples ci-dessus pour \'ecrire la contragr\'ediente de
${\rm Ind}_{\BB(\A)}^{\GG(\A)}{\rm LP}(\LL(\wZ))$ comme l'ensemble des
vecteurs $\GG(\wZ)$-finis de 
$$\prod\nolimits_{r,s,\eta_1,\eta_2}
 \big({\rm Ind}_{\BB(\A)}^{\GG(\A)}
\big(\eta_1\,|\ |_{\A,p}^r\otimes\eta_2\,|\ |_{\A,p}^s\big)\big)^\vee$$
Pour conclure, il suffit d'utiliser le fait que l'ensemble des vecteurs $\GG(\wZ)$-finis du
produit est la somme directe car il n'y a qu'un nombre fini de couples $(\eta_1,\eta_2)$
de conducteurs born\'es et que des couples $(r,s)$ distincts induisent des repr\'esentations
alg\'ebriques irr\'eductibles et non isomorphes.
\end{proof}

\begin{rema}\phantomsection\label{indu6.8}
La d\'ecomposition~(\ref{indu6.7}) fournit une 
d\'ecomposition
$${\rm Ind}_{\BB(\A)}^{\GG(\A)}{\rm LP}(\LL(\wZ))=\bigoplus_{r,s,\eta_1,\eta_2} 
{\rm Ind}_{\BB(\A)}^{\GG(\A)} \big(\eta_1|\ |_{\A,p}^{r}\otimes \eta_2|\ |_{\A,p}^{s}\big)$$
mais il r\'esulte des rem.\,\ref{indu1.1} et~\ref{indu3} que l'induite
\`a droite est $0$ si $s<r$ ou si $(-1)^{r+s}\eta_{1,\infty}\eta_{2,\infty}(-1)=-1$
(o\`u $\eta_{1,\infty},\eta_{2,\infty}$ sont les restrictions de $\eta_1,\eta_2$ \`a $\R^\dual$).
\end{rema}

\section{Le cas $\GG={\bf GL}_2$}

%{\tiny
%\Subsubsection{Le groupe $\GG(\A)$}\label{COCO1}
%Soit $\A$ l'anneau des ad\`eles de $\Q$.
%On a $\A=\R\times\A_f$, avec $\A_f=\Q\otimes\wZ $.
%Si $p$ est un nombre premier, on a une d\'ecomposition
%de $\A$ sous la forme $\A=\Q_p\times \A^{]p[}=
%\Q_p\times\R\times \A^{]p[}_f$, o\`u
%$\A^{]p[}_f=\Q\otimes\wZp$ et $\wZp=\prod_{\ell\neq p}\Z_\ell$.
%On peut \'ecrire $x\in\A$, en utilisant les d\'ecompositions pr\'ec\'edentes,
%sous les formes suivantes:
%$$x=(x_\infty,x_f)=(x_p,x^{]p[})=(x_\infty,x_p,x_f^{]p[}).$$
%
%
%Soit $G={\bf GL}_2$.  
%Les projections naturelles $\GG(\A)\to \GG(\Q_v)$ ou $\GG(\A)\to \GG(\A_f)$ admettent
%des sections naturelles (en compl\'etant avec des $1$), ce qui permet de
%voir $\GG(\Q_v)$ et $\GG(\A_f)$ comme des sous-groupes de $\GG(\A)$.
%
%On note $\GG(\R)_+$ la composante connexe de l'identit\'e dans $\GG(\R)$;
%c'est l'ensemble des $x\in \GG(\R)$ avec $\det x>0$.
%
%On plonge $\GG(\Q)$ diagonalement dans $\GG(\A)$.
%Si $\gamma\in \GG(\Q)$, on note $\gamma_\infty$ sa composante
%\`a l'infini et $\gamma_f$ son image dans $\GG(\A_f)$.
%On a donc $\gamma=\gamma_\infty\gamma_f=\gamma_f\gamma_\infty$ dans
%$\GG(\A)$.
%}

\Subsection{Descente de $\GG(\Q)$ \`a $\Gamma$}\label{COCO3}
On note $\Gamma$ le sous-groupe $\GG(\Z)_+={\bf SL}_2(\Z)$ de $\GG(\Q)$.
Le r\'esultat suivant est classique.
\begin{lemm}\phantomsection\label{coco1}
Tout \'el\'ement $x$ de $\GG(\A)$ peut s'\'ecrire sous la forme
$x=\gamma^{-1}x_\infty\kappa$, avec
$\gamma\in \GG(\Q)$, $x_\infty\in \GG(\R)_+$ et $\kappa\in \GG(\widehat \Z)$.
De plus $\gamma^{-1}x_\infty\kappa=(\gamma')^{-1}x'_\infty\kappa'$ si
et seulement s'il existe $\alpha\in \Gamma$ tel
que $\gamma'=\alpha\gamma$, $x'_\infty=\alpha_\infty x_\infty$
et $\kappa'=\alpha^{]\infty[}\kappa$.
\end{lemm}

On fait agir $\Gamma$ sur ${\cal C}(\GG(\wZ ))$ par
$(\alpha\cdot\phi)(x)=\phi((\alpha^{]\infty[})^{-1}x)$.
\begin{lemm}\phantomsection\label{coco4}
Si $X={\cal C},{\cal C}^{(p)},{\rm LA},{\rm LP},{\rm LC}$,
alors
$X(\GG(\A))\cong {\rm Ind}_\Gamma^{\GG(\Q)}X(\GG(\wZ) )$.
\end{lemm}
\begin{proof}
Si $\phi\in {\cal C}(\GG(\A),L)$, alors $\phi$ est constante sur les classes
modulo $\GG(\R)_+$. Si $\gamma\in \GG(\Q)$, soit $\phi_\gamma\in {\cal C}(\GG(\wZ),L)$
d\'efinie par $\phi_\gamma(\kappa)=\phi(\gamma^{-1} x_\infty\kappa)$ pour n'importe quel
choix de $x_\infty\in \GG(\R)_+$ (le r\'esultat ne d\'epend pas de ce choix d'apr\`es ce qui pr\'ec\`ede).

Si $\alpha\in\Gamma$, alors
\begin{align*}
\phi_{\alpha\gamma}(\kappa)=\phi(\gamma^{-1}\alpha^{-1}x_\infty\kappa)&=
\phi(\gamma^{-1}(\alpha_\infty^{-1}x_\infty)((\alpha^{]\infty[})^{-1}\kappa))\\ 
&= \phi_\gamma((\alpha^{]\infty[})^{-1}\kappa)=(\alpha*\phi_\gamma)(\kappa),
\end{align*}
ce qui prouve que $(\phi_\gamma)_{\gamma\in \GG(\Q)}\in {\rm Ind}_\Gamma^{\GG(\Q)}{\cal C}(\GG(\wZ))$.

R\'eciproquement, si $(\phi_\gamma)_{\gamma\in \GG(\Q)}\in {\rm Ind}_\Gamma^{\GG(\Q)}{\cal C}(\GG(\wZ))$,
on d\'efinit $\phi\in {\cal C}(\GG(\A),L)$, en posant $\phi(\gamma^{-1}x_\infty\kappa)=
\phi_\gamma(\kappa)$: le lemme~\ref{coco1} implique que 
$\gamma^{-1} x_\infty \kappa$, avec $\gamma\in \GG(\Q)$,
la condition
$\phi_{\alpha\gamma}=\alpha*\phi_\gamma$ est exactement celle qu'il faut pour
que $\phi(z)$ ne d\'epende pas de la d\'ecomposition de $z$ sous
la forme $\gamma^{-1}x_\infty\kappa$.
\end{proof}

Il r\'esulte du lemme~\ref{coco4} et du lemme de Shapiro
que l'on a:
\begin{coro}\phantomsection\label{coco5}
Si $X={\cal C},{\cal C}^{(p)},{\rm LA},{\rm LP},{\rm LC}$ et si $i\in\N$,
alors
$$H^i(\GG(\Q),X(\GG(\A)))\cong H^i(\Gamma,X(\GG(\wZ))).$$
\end{coro}

Sous la forme du membre de droite, 
les groupes ont l'air beaucoup plus petits et maniables, mais l'action de
$\GG(\A)$ n'est plus visible.
\begin{rema}
Il r\'esulte du cor.\,\ref{coco5} et de la compacit\'e de $\GG(\wZ)$ que
$$H^1(\GG(\Q),{\cal C}(\GG(\A)))=H^1(\GG(\Q),{\cal C}_b(\GG(\A)))$$
o\`u ${\cal C}_b$ d\'esigne l'espace des fonctions continues born\'ees
(cf.~(i) de la prop.\,\ref{coco7}).
\end{rema}

\begin{prop}\phantomsection\label{coco6}
{\rm (i)} 
Si $X={\cal C},{\cal C}^{(p)},{\rm LA},{\rm LP},{\rm LC}$,
alors,
$$H^0(\GG(\Q),X(\GG(\A)))\cong X(\wZ ^\dual)
\quad{\rm et}\quad
H^i(\GG(\Q),X(\GG(\A)))=0\ {\text{si $i\geq 2$}}.$$
\end{prop}
\begin{proof}
Le cas $i=0$ est une cons\'equence 
du cor.\,\ref{coco5}
et de ce que
l'adh\'erence de ${\bf SL}_2(\Z)$ dans $\GG(\wZ )$ est
${\bf SL}_2(\wZ )$:
une fonction continue de $g\in \GG(\wZ)$, invariante par
$\Gamma$, ne d\'epend donc que de $\det g\in\wZ^\dual$.

Le cas $i\geq 2$ est une cons\'equence du fait que $\Gamma$ contient un sous-groupe
libre d'indice fini.
\end{proof}

\subsubsection{Cohomologie de $\Gamma$}\label{gamma}
Soient 
$$I=\matrice{-1}{0}{0}{-1},
\quad S=\matrice{0}{1}{-1}{0},\quad U=\matrice{0}{1}{-1}{1}.$$
\begin{lemm}\phantomsection\label{coco2}
Si $M$ est un $\Gamma$-module sur lequel $2$ est inversible,
on a une suite exacte:
$$\xymatrix{0\to H^0(\Gamma ,M)\ar[r]&
M^{S=1}\ar[r]^-{1-U}& M^{1+U+U^2=0}\ar[r]
&H^1(\Gamma ,M)\to0}$$
\end{lemm}
\begin{proof}
$U^3=S^2=I$ et $I^2=1$ est une pr\'esentation
de $\Gamma$ par g\'en\'erateurs et relations.
En particulier, $S$ et $U$ sont respectivement d'ordre
$2$ et $3$ dans $\overline\Gamma=\Gamma/\{\pm I\}$, et 
$\overline\Gamma $ est le produit libre des sous-groupes
engendr\'es par $S$ et $U$.

Soit 
$M_+=\big\{x\in M,\ I\cdot x=x\big\}$.
Alors $H^i(\Gamma,M)=H^i(\overline\Gamma ,M_+)$.
Par ailleurs, comme $\overline\Gamma $ est le produit libre
des sous-groupes
engendr\'es par $S$ et $U$, un $1$-cocycle sur
$\overline\Gamma $ est d\'etermin\'e par ses valeurs
$x$ et $y$ en $S$ et $U$, avec les seules conditions
$x+Sx=0$ et $(1+U+U^2)y=0$.
On en d\'eduit que
$$H^1(\overline\Gamma ,M_+)\cong
\frac{\{(x,y)\in M_+\times M_+,\ x+Sx=0,\ (1+U+U^2)y=0\}}{\{(1-S)z,(1-U)z,\ z\in M_+\}}\,.$$
En prenant $z=\frac{1}{2}x$, on peut tuer $x$ et obtenir la suite exacte suivante
\begin{equation}\phantomsection\label{gamma1}
\xymatrix{0\to H^0(\overline\Gamma ,M_+)\ar[r]&
M_+^{S=1}\ar[r]^-{1-U}& M_+^{1+U+U^2=0}\ar[r]
&H^1(\overline\Gamma ,M_+)\to0}
\end{equation}
On conclut en remarquant que
$M^{S=1}=M_+^{S=1}$ et $M^{1+U+U^2=0}=M_+^{1+U+U^2=0}$.
\end{proof}

\Subsection{La cohomologie des fonctions continues}
\Subsubsection{Lien avec la cohomologie compl\'et\'ee d'Emerton}\label{COCO4}
On envoie $\C^*$ dans $\GG(\R)$ par $a+ib\mapsto\matrice{a}{-b}{b}{a}$.
Alors $$\GG(\R)/\C^\dual\cong {\bf P}^1(\C)-{\bf P}^1(\R)={\cal H}\sqcup\overline{\cal H}.$$
Soit $\wGamma(N)={\rm Ker}(\GG(\wZ)\to \GG(\Z/N\Z))$.
Soit $$Y(N)_\C=\GG(\Q)\backslash \GG(\A)/(\C^\dual\times\wGamma(N)).$$

La proposition suivante montre que $H^1(\GG(\Q),{\cal C}(\GG(\A),L))$ n'est
autre que la cohomologie compl\'et\'ee d'Emerton de la tour des courbes
modulaires de tous niveaux.

\begin{prop}\phantomsection\label{coco7}
{\rm (i)} $H^1(\GG(\Q),{\cal C}(\GG(\A),L))=L\otimes_{\O_L}H^1(\GG(\Q),{\cal C}(\GG(\A),\O_L))$.

{\rm (ii)} $H^1(\GG(\Q),{\cal C}(\GG(\A),\O_L))=\varprojlim_k 
H^1(\GG(\Q),{\cal C}(\GG(\A),\O_L/p^n))$.

{\rm (iii)}
$H^1(\GG(\Q),{\cal C}(\GG(\A),\O_L/p^n))\cong 
H^1(\Gamma,{\cal C}(\GG(\wZ),\O_L/p^n))$.

{\rm (iv)} $H^1(\Gamma,{\cal C}(\GG(\wZ),\O_L/p^n))
=\varinjlim_N H^1(\Gamma,{\cal C}(\GG(\Z/N\Z),\O_L/p^n))$.

{\rm (v)}
$H^1(\Gamma,{\cal C}(\GG(\Z/N\Z),\O_L/p^n))\cong H^1(Y(N)_\C,\O_L/p^n)$.
\end{prop}
\begin{proof}
Pour le (i), on utilise les isomorphismes
\begin{align*}
H^1(\GG(\Q),{\cal C}(\GG(\A),L))&\cong H^1(\Gamma,{\cal C}(\GG(\wZ),L))\\
H^1(\GG(\Q),{\cal C}(\GG(\A),\O_L))&\cong H^1(\Gamma,{\cal C}(\GG(\wZ),\O_L))
\end{align*}
la compacit\'e de $\GG(\wZ)$ qui implique que
${\cal C}(\GG(\wZ),L)\cong L\otimes_{\O_L}{\cal C}(\GG(\wZ),\O_L)$,
et le 
lemme~\ref{coco2}.

Le (ii) est une cons\'equence facile
du lemme~\ref{coco2} et de ce que
\begin{align*}
{\cal C}(\GG(\wZ),\O_L)^{1+U+U^2}/p^n&\to {\cal C}(\GG(\wZ),\O_L/p^n)^{1+U+U^2}\\
{\cal C}(\GG(\wZ),\O_L)^{S=1}/p^n&\to {\cal C}(\GG(\wZ),\O_L/p^n)^{S=1}
\end{align*}
sont injectives de conoyaux respectivement tu\'es par $6$ et $4$
puisque $U$ et $S$ sont d'ordre $6$ et $4$.

Le (iii) r\'esulte de ce que
${\cal C}(\GG(\wZ),\O_L/p^n)=\varinjlim_N {\cal C}(\GG(\Z/N\Z),\O_L/p^n)$.

Le (v) est une cons\'equence du lemme de
Shapiro et de ce que le demi-plan sup\'erieur est contractile.
\end{proof}

\begin{rema}\phantomsection\label{coco7.3}
On a aussi $H^1(Y(N)_\C,\O_L/p^n)\cong H^1_{\eet}(Y(N)_{\Qbar},\O_L/p^n)$ et comme
$Y(N)$ est d\'efinie sur $\Q$, cela munit tous les goupes ci-dessus d'une action
du groupe de Galois absolu de $\Q$.
\end{rema}

\begin{rema}\phantomsection\label{cup7.3}
Soit $S\subset{\cal P}\cup\{\infty\}$ fini.
Si $\infty\in S$ (resp.~$\infty\notin S$),
 tout \'el\'ement
de $\GG(\Q_S)$ s'\'ecrit sous la forme $\gamma^{-1}x_\infty\kappa$ (resp.~$\gamma^{-1}\kappa$),
avec $\gamma\in \GG(\Z[\frac{1}{S}])$, $x_\infty\in \GG(\R)_+$ et $\kappa\in \GG(\Z_S)$,
et une telle \'ecriture est unique \`a changements
simultan\'es $\gamma\mapsto \alpha\gamma$, $x_\infty\mapsto \alpha_\infty x_\infty$
et $\kappa\mapsto \alpha_{S\moins\{\infty\}}\kappa$, avec $\alpha\in\Gamma$ (resp.~$\alpha\in \GG(\Z)$).
On en d\'eduit des isomorphismes de $\GG(\Z[\frac{1}{S}])$-modules, 
si $X={\cal C}$, ${\cal C}^{(p)}$, ${\rm LC}$, ${\rm LP}$, ${\rm LA}$:
$$ X(\GG(\Q_S),L)=\begin{cases}
{\rm Ind}_{\Gamma}^{\GG(\Z[\frac{1}{S}])}X(\GG(\Z_S),L) &{\text{si $\infty\in S$,}}\\
{\rm Ind}_{\GG(\Z)}^{\GG(\Z[\frac{1}{S}])}X(\GG(\Z_S),L) &{\text{si $\infty\notin S$}}
\end{cases}$$
En particulier, si $\infty\in S$, alors
$H^1(\GG(\Z[\frac{1}{S}]),{\cal C}(\GG(\Q_S),L))$ est la cohomologie compl\'et\'ee
de la tour des courbes modulaires de niveaux dont les facteurs premiers appartiennent \`a $S$.
\end{rema}

\subsubsection{Vecteurs lisses en dehors de $p$}
L'action de $\GG(\wZp)$ sur ${\cal C}^{(p)}(\GG(\A))$ n'est pas lisse (seulement localement
lisse) mais
la proposition suivante montre que celle sur la cohomologie
de ${\cal C}^{(p)}(\GG(\A))$ l'est.
\begin{prop}\phantomsection\label{cp}
L'application naturelle
$$\varinjlim\nolimits_{K^{]p[}} H^1(\GG(\Q),{\cal C}(\GG(\A)/{K^{]p[}}))\to H^1(\GG(\Q),{\cal C}^{(p)}(\GG(\A)))$$
est un isomorphisme {\rm (${K^{]p[}}$ d\'ecrit les sous-groupes ouverts de $\GG(\wZ^{]p[})$)}
\end{prop}
\begin{proof}
On a une suite d'isomorphismes
\begin{align*}
H^1(\GG(\Q),{\cal C}^{(p)}(\GG(\A)))&\cong H^1(\Gamma,{\cal C}^{(p)}(\GG(\wZ)))\\
&\cong \varinjlim\nolimits_{K^{]p[}} H^1(\Gamma,{\cal C}(\GG(\wZ)/{K^{]p[}}))\\
&\cong \varinjlim\nolimits_{K^{]p[}} H^1(\GG(\Q),{\cal C}(\GG(\A)/{K^{]p[}}))
\end{align*}
Le premier isomorphisme est le lemme de Shapiro, le second r\'esulte de ce que
${\cal C}^{(p)}(\GG(\wZ)))\cong \varinjlim\nolimits_{K^{]p[}}{\cal C}(\GG(\wZ)/{K^{]p[}})$ (par compacit\'e de
$\GG(\wZ)$), et le troisi\`eme est de nouveau le lemme de Shapiro.

Le r\'esultat s'en d\'eduit.
\end{proof}

\begin{rema}\phantomsection\label{cp1}
(i) Nous prouvons plus loin (prop.\,\ref{coco14}) que
$$H^1(\GG(\Q),{\cal C}(\GG(\A)/{K^{]p[}}))=H^1(\GG(\Q),{\cal C}(\GG(\A)))^{K^{]p[}}$$
et donc que $H^1(\GG(\Q),{\cal C}^{(p)}(\GG(\A)))$ est l'ensemble des
vecteurs $\GG(\wZp)$-lisses de $H^1(\GG(\Q),{\cal C}(\GG(\A)))$.

(ii) En reprenant les arguments de la preuve de la prop.\,\ref{coco7}, on d\'emontre
que $H^1(\GG(\Q),{\cal C}(\GG(\A)/{K^{]p[}}))$ est la cohomologie compl\'et\'ee
de la tour des courbes modulaires de niveaux $K^{]p[}K_p$ o\`u $K_p$
d\'ecrit les sous-groupes ouverts de $\GG(\Z_p)$.
\end{rema}

\begin{prop}\phantomsection\label{coco15}
$H^1(\GG(\Q),{\cal C}(\GG(\A)/{K^{]p[}}))$ est une repr\'esentation admissible de $\GG(\Q_p)$.
\end{prop}
\begin{proof}
Par d\'efinition, il suffit de regarder l'action de $\GG(\Z_p)$, 
et on peut utiliser le lemme de Shapiro pour descendre de $\GG(\Q)$ \`a $\Gamma$
et \'ecrire le groupe qui nous int\'eresse comme
$H^1(\Gamma,M)$, avec $M={\cal C}(\GG(\Z_p)\times \GG(\wZp)/K^{]p[})$.
Le $\GG(\Z_p)$-module $M$ est trivialement admissible
(c'est la somme directe d'un nombre fini de copies de ${\cal C}(\GG(\Z_p))$).
On en d\'eduit que $M^{S=1}$ et $M^{1+U+U^2}$ sont admissibles
en tant que sous-$\GG(\Z_p)$-modules ferm\'es
d'un module admissible.
Le lemme~\ref{coco11} 
permet alors d'en d\'eduire que
$H^0(\Gamma,M)$ et $H^1(\Gamma,M)$ sont admissibles en tant
que noyau et conoyau d'un morphisme de $\GG(\Z_p)$-modules admissibles~\cite{ST1}.
\end{proof}

\begin{prop}\phantomsection\label{coco14}
On a
$$H^1(\GG(\Q),{\cal C}(\GG(\A)/{K^{]p[}}))=H^1(\GG(\Q),{\cal C}(\GG(\A)))^{K^{]p[}}$$
\end{prop}
\begin{proof}
Posons $K=K^{]p[}$.
Le lemme de Shapiro nous ram\`ene \`a prouver que
$H^1(\Gamma,M^K)=H^1(\Gamma,M)^K$, avec $M={\cal C}(\GG(\wZ))$.
Pour cela, compte-tenu du calcul des $H^0$ du lemme~\ref{coco13},
il suffit de prouver que la suite des $K$-invariants de la suite
du lemme~\ref{coco2} est exacte.  
Cela r\'esulte du crit\`ere du lemme~\ref{coco11}
avec $A={\cal C}(\wZ^\dual)$ et $B={\cal C}(\GG(\wZ^\dual))^{S=1}$,
et les r\'esultats d'annulation des lemmes~\ref{coco13} et~\ref{coco12}.
\end{proof}

\begin{lemm}\phantomsection\label{coco13}
On a les r\'esultats suivants:
\begin{align*}
&H^0(K,{\cal C}(\GG(\wZ))^{S=1})={\cal C}(\GG(\wZ)/K)^{S=1}
&&H^1(K,{\cal C}(\GG(\wZ))^{S=1})=0\\
&H^0(K,{\cal C}(\GG(\wZ))^{1+U+U^2=0})={\cal C}(\GG(\wZ)/K)^{1+U+U^2=0}
&&H^1(K,{\cal C}(\GG(\wZ))^{1+U+U^2=0})=0
\end{align*}
\end{lemm}
\demo
Comme $2$ et $3$ sont inversibles dans $L$, 
on a des d\'ecompositions $K$-\'equivariantes
\begin{align*}
{\cal C}(\GG(\wZ))&={\cal C}(\GG(\wZ))^{S=1}\oplus {\cal C}(\GG(\wZ))^{S=-1}\\
{\cal C}(\GG(\wZ))&={\cal C}(\GG(\wZ))^{1+U+U^2=0}\oplus {\cal C}(\GG(\wZ))^{U=1}
\oplus {\cal C}(\GG(\wZ))^{U^3=-1}
\end{align*}
Il suffit de prouver que
$H^0(K,{\cal C}(\GG(\wZ)))={\cal C}(\GG(\wZ)/K)$, ce qui est imm\'ediat,
et $H^1(K,{\cal C}(\GG(\wZ)))=0$, ce qui suit de ce que
${\cal C}(\GG(\wZ))={\cal C}(\GG(\Z_p))\widehat\otimes{\cal C}(\GG(\wZp))$
et ${\cal C}(\GG(\wZp))\cong{\cal C}(K)^{\GG(\wZp)/K}$.

\begin{lemm}\phantomsection\label{coco11}
Soit $0\to A\to B\to C\to D\to 0$
une suite exacte de $K$-modules topologiques.
Si $H^1(K,B)=0$ et si $H^1(K,A)=H^2(K,A)=0$, la suite
$0\to A^K\to B^K\to C^K\to D^K\to 0$ est exacte.
\end{lemm}
\begin{proof}
On a des suites exactes $0\to A\to B\to B/A\to 0$ et $0\to B/A\to C\to D$.
Si $H^1(K,A)=0$, on a $(B/A)^K=B^K/A^K$, et donc
la suite $0\to A^K\to B^K\to C^K\to D^K\to 0$ est exacte
si et seulement si $C^K\to D^K$ est surjective.
Ceci est le cas si $H^1(K,B/A)=0$.
Or on a une suite exacte $H^1(K,B)\to H^1(K,B/A)\to H^2(K,A)$;
l'hypoth\`ese $H^1(K,B)=H^2(K,A)=0$ permet donc de conclure.
\end{proof}

\subsubsection{R\'esultats d'annulation pour la cohomologie de ${\bf SL}_2$}
Soit $\GG'={\bf SL}_2\subset\GG$.

\begin{lemm}\phantomsection\label{psl2}
On suppose $\ell\neq p$.

{\rm (i)} $H^1(\GG'({\bf F}_\ell),\Z/p^n)=0$. 

{\rm (ii)} $H^2(\GG'({\bf F}_\ell),\Z/p^n)=0$.
\end{lemm}
\begin{proof}
Le (i) est \'evident si $\ell\geq 5$ car ${\bf PSL}_2({\bf F}_\ell)$
est alors un groupe simple (et $H^1={\rm Hom}$). 
Les cas $\ell=2,3$ se traitent \`a la main et il ressort que, m\^eme
si on ne suppose pas $\ell\neq p$, les seuls cas o\`u ${\rm Hom}(\GG'({\bf F}_\ell),\Z/p^n)\neq 0$
sont:
$p=2$ et $\ell=2$ o\`u ce groupe est \'egal \`a $\Z/2$,
$p=3$ et $\ell=3$ o\`u ce groupe est \'egal \`a $\Z/3$.

Le (ii)
est un cas particulier du th.\,14 du chap.\,7 de
 \cite{Steinberg} (cf. (a) des Remarks \`a la suite du th\'eor\`eme).
\end{proof}

\begin{lemm}\phantomsection\label{coco12}
{\rm (i)} Si $K_N={\rm Ker}\big(\GG'(\wZp)\to \GG'(\Z/N)\big)$, avec $(N, p)=1$,
alors $H^i(K_N,\Z_p)=0$ si $i=1,2$.

{\rm (ii)} Si $K$ est un sous-groupe ouvert de $\GG(\wZp)$,
alors $H^i(K,{\cal C}(\wZ))=0$ si $i=1,2$.
\end{lemm}
\begin{proof} (i)

$\bullet$ Pour $H^1$, on commence par remarquer
que tout morphisme continu $K_N\to \Z/p^m$ se factorise \`a travers
$\prod_{\ell\nmid pN}\GG'({\bf F}_\ell)$ 
car le noyau est un produit de pro-$\ell$-groupes, avec
$\ell\neq p$. Comme cette image est un produit de groupes finis et que $\Z_p$
ne contient pas de sous-groupe fini, cela permet de
conclure. (En fait, il r\'esulte du (i) du lemme~\ref{psl2} que  
$H^1(K_N,\Z/p^m)=0$ 
pour tout $m$.)

$\bullet$ Pour $H^2$, la nullit\'e de
$H^1(K_N,\Z/p^m)$ implique
que $H^2(K_N,\Z_p)$ est la limite projective des
$H^2(K_N,\Z/p^n)$ puisque le ${\rm R}^1\lim$ des $H^1(K_N,\Z/p^m)$ est nul.
Il suffit donc de prouver que $H^2(K_N,\Z/p^n)=0$ 
pour tout $n$ et, par d\'evissage, il suffit de le prouver pour $n=1$. 

Comme le noyau de $K_N\to \prod_{\ell\nmid pN}\GG'({\bf F}_\ell)$
est d'ordre premier \`a $p$, la suite spectrale de Hochschild-Serre permet
de remplacer $K_N$ par $K'_N=\prod_{\ell\nmid pN}\GG'({\bf F}_\ell)$.
 Par continuit\'e des cocycles consid\'er\'es,
$H^2(K'_N,\Z/p^n)$ est la limite inductive
des $H^2(\prod_{\ell\in S}\GG'({\bf F}_\ell),\Z/p)$, avec $S$ fini.
Une r\'ecurrence imm\'ediate sur le cardinal de $S$,
utilisant la suite spectrale de Hochschild-Serre et le lemme~\ref{psl2} pour $|S|=1$,
permet de conclure.

(ii)
On a une suite exacte
$1\mapsto K'\to K\to U\to 1$, o\`u $K'$ est un sous-groupe
ouvert de $\GG'(\wZp)$ et $U$ un sous-groupe ouvert
de $\wZ^{]p[,\dual}$.
Quitte \`a remplacer $K$ par un sous-groupe d'indice fini 
(et \`a utiliser restriction et corestriction
pour revenir \`a $K$), on peut supposer que 
$K'=K_N$.
On est alors ramen\'e \`a v\'erifier que
$H^i(U,H^j(K_N,{\cal C}(\wZ^\dual)))=0$ si $1\leq i+j\leq 2$
(en particulier $j\leq 2$).

Comme $K$ agit par $g\cdot\phi(x)=\phi(x\det g)$ sur
${\cal C}(\wZ^\dual)$, l'action de $K_N$ est triviale
et donc $H^0(K_N,{\cal C}(\wZ^\dual))={\cal C}(\wZ^\dual)$.
Maintenant, ${\cal C}(\wZ^\dual)={\cal C}(\Z_p^\dual)\widehat\otimes
{\cal C}(\wZ^{]p[,\dual})$, et on a $H^i(U,{\cal C}(\wZ^{]p[,\dual}))=0$
si $i\geq 1$, pour tout sous-groupe ouvert
compact $U$ de $\wZ^{]p[,\dual}$.
On en d\'eduit que
$H^i(U,{\cal C}(\wZ^\dual))=0$ pour tout
$i\geq 1$ (car $U$ agit trivialement sur ${\cal C}(\Z_p^\dual)$),
ce qui permet de traiter le cas $j=0$.

Si $j=1,2$,
comme l'action de $K_N$ sur ${\cal C}(\wZ^\dual)$ est triviale, 
on a $H^j(K_N,{\cal C}(\wZ^\dual))=H^j(K_N,\Z_p)\widehat\otimes
{\cal C}(\wZ^\dual)=0$, d'apr\`es le (i).

Ceci permet de conclure.
\end{proof}

\begin{lemm}\phantomsection\label{coco18}
{\rm(cf.~\cite[cor.\,4.3.2]{Em06})}
Si ${K_p}$ est un sous-groupe ouvert de $\GG'(\Z_p)$,
et si $W$ est une repr\'esentation de dimension
finie de $\GG(\Q_p)$, alors
$H^i({K_p},W)=0$ si $i=1,2$.
\end{lemm}
\begin{proof}
D'apr\`es Lazard~\cite{laz}, $H^i({K_p},W)\cong H^i({\goth{sl}}_2,W)$,
et comme ${\goth{sl}}_2$ est une alg\`ebre de Lie semi-simple,
on a $H^1({\goth{sl}}_2,W)$ 
et $H^2({\goth{sl}}_2,W)=0$ d'apr\`es les premier et second lemmes de
Whitehead.
\end{proof}

\begin{lemm}\phantomsection\label{coco200}
Si $K$ est un sous-groupe ouvert de $\GG(\wZ)$ et si $W$ est une repr\'esentation alg\'ebrique
de $\GG(\Q_p)$, alors $H^i(K,{\cal C}(\GG(\wZ))\otimes W)=0$ si $i=1,2$.
\end{lemm}
\begin{proof}
Si $K'=K\cap \GG'(\wZ)$, alors ${\cal C}(\GG(\wZ))\otimes W$
est la somme d'un nombre fini de copies de ${\rm Ind}_{K'}^KW$, et il suffit
de prouver que $H^i(K',W)=0$ si $i=1,2$.

Quitte \`a remplacer notre $K$ initial par un sous-groupe d'indice fini (et utiliser la
corestriction pour revenir \`a $K$), on peut supposer $K'=K^{]p[}K_p$ o\`u $K_p$
est un sous-groupe ouvert de $\GG(\Z_p)$ et $K^{]p[}$ est un sous-groupe ouvert de $\GG(\wZp)$.
Alors $H^i(K',W)=\oplus_{j+k=i}H^j(K^{]p[},L)\otimes H^k(K_p,W)$ et dans chaque terme,
au moins un des deux groupes qui appara\^{\i}t est $0$ d'apr\`es les lemmes~\ref{coco12} et~\ref{coco18}.
\end{proof}

\Subsection{Vecteurs localement analytiques et localement alg\'ebriques}\label{COCO7}
Si $W$ est une repr\'esentation de $\GG(\Q_p)$, on note $W^{\rm an}$ et $W^{\rm alg}$
les sous-espaces des vecteurs localement analytiques et localement alg\'ebriques.
\Subsubsection{Cohomologie des fonctions localement analytiques}
\begin{theo}\phantomsection\label{coco16}
$H^1(\GG(\Q),{\cal C}^{(p)}(\GG(\A)))^{\rm an}=
H^1(\GG(\Q),{\rm LA}(\GG(\A)))$.
\end{theo}
\demo
C'est une question qui ne fait intervenir que l'action de $\GG(\Z_p)$.
Le lemme de Shapiro permet de descendre de $\GG(\Q)$ \`a $\Gamma$ sans faire disparaitre
l'action de $\GG(\Z_p)$:
il suffit de prouver que, pour tout quotient $Y=\GG(\wZp)/K^{]p[}$, o\`u $K^{]p[}$ est un sous-groupe
ouvert de $\GG(\wZp)$, on a
$$H^1(\Gamma,{\cal C}(\GG(\Z_p)\times Y))^{\rm an}=
H^1(\Gamma,{\rm LA}(\GG(\Z_p)\times Y)).$$
Cela r\'esulte de: 

$\bullet$ la suite exacte du lemme~\ref{coco2}
appliqu\'ee \`a $M={\cal C}(\GG(\Z_p)\times Y)$,

$\bullet$ l'exactitude du foncteur $\Pi\mapsto\Pi^{\rm an}$ pour les repr\'esentations
admissibles~\cite{ST4},

$\bullet$ l'identit\'e
${\cal C}(\GG(\Z_p)\times Y)^{\rm an}={\rm LA}(\GG(\Z_p)\times Y)$,

$\bullet$ la suite exacte du lemme~\ref{coco2}
appliqu\'ee \`a $M={\rm LA}(\GG(\Z_p)\times Y)$.

\Subsubsection{Vecteurs localement alg\'ebriques de la cohomologie compl\'et\'ee}\label{COCO8}
Le r\'esultat suivant est une reformulation d'un r\'esultat d'Emerton~\cite[(4.3.4)]{Em06},
\cite[th.\,7.4.2]{Em06b}
 (et la preuve qui suit est une adaptation de celle d'Emerton).
\begin{theo}\phantomsection\label{coco17}
$H^1(\GG(\Q),{\cal C}^{(p)}(\GG(\A)))^{\rm alg}=
H^1(\GG(\Q),{\rm LP}(\GG(\A)))$.
\end{theo}
\begin{proof}
Via le lemme de Shapiro, on se ram\`ene, comme ci-dessus, \`a prouver
que $H^1(\Gamma,{\cal C}(\GG(\Z_p)\times Y))^{\rm alg}=
H^1(\Gamma,{\cal C}(\GG(\Z_p)\times Y)^{\rm alg})$, pour tout quotient $Y=\GG(\wZp)/K^{]p[}$.
Ceci se ram\`ene \`a v\'erifier que, si $W$ est une repr\'esentation alg\'ebrique
de $\GG(\Q_p)$, et si $K_p$ est un sous-groupe ouvert de $\GG(\Z_p)$,
alors
$$H^1(\Gamma,{\cal C}(\GG(\Z_p)\times Y)\otimes W)^{K_p}=
H^1(\Gamma,({\cal C}(\GG(\Z_p)\times Y)\otimes W)^{K_p}).$$
Il suffit donc de prouver que la suite des ${K_p}$-invariants
de la suite exacte du lemme~\ref{coco2} tensoris\'ee par $W$
est encore exacte.
Pour cela, on utilise le crit\`ere du lemme~\ref{coco11},
avec $A={\cal C}(\Z_p^\dual)\otimes W$
et $B={\cal C}(\GG(\Z_p))\otimes W$.
On a $H^1({K_p},B)=0$ car
$B$ est isomorphe \`a une somme de copies
de ${\cal C}({K_p})$, et le lemme~\ref{coco200} fournit la nullit\'e
de $H^1({K_p},A)$ et $H^2({K_p},A)$ qui permet de conclure.
\end{proof}

En utilisant la d\'ecomposition~(\ref{decomp}) de ${\rm LP}(\GG(\A))$, on en d\'eduit le r\'esultat
suivant.
\begin{coro}\phantomsection\label{coco21}
On a une d\'ecomposition
$$H^1(\GG(\Q),{\cal C}^{(p)}(\GG(\A)))^{\rm alg}=
\bigoplus\nolimits_W\big(H^1(\GG(\Q),{\rm LC}(\GG(\A))\otimes W)\big)\otimes W^\dual$$
o\`u $W$ parcourt les $\Q_p$-repr\'esentations alg\'ebriques irr\'eductibles de $\GG$.
\end{coro}

\begin{rema}\phantomsection\label{coco22}
(i)
L'action de $\GG(\Q_p)$ sur $H^1(\GG(\Q),{\rm LC}(\GG(\A))\otimes W)$ est lisse;
celle sur $W^\dual$ est alg\'ebrique.

(ii) Le lemme de Shapiro permet de descendre de $\GG(\Q)$ \`a $\Gamma$ puis \`a $\Gamma(N)$
pour les fonctions invariantes par $\wGamma(N):={\rm Ker}(\GG(\wZ)\to\GG(\Z/N))$: on obtient alors
$$H^1(\GG(\Q),{\rm LC}(\GG(\A))\otimes W)\cong\varinjlim\nolimits_{N}
H^1_{\rm et}(Y_{N,\Qbar},W)$$
ou $N$ parcourt les entiers~$\geq 1$.
L'espace $H^1_{\rm et}(Y_{N,\Qbar},W)$ est celui dans lequel on
decoupe les repr\'esentations galoisiennes associ\'ees aux formes modulaires de niveau $N$
et de poids d\'etermin\'e par $W$.
\end{rema}

\Subsection{Cohomologie \`a support compact}\label{cup1}
\subsubsection{Cocycles nuls sur le borel}
Si $M$ est un $\GG(\Q)$-module,
on d\'efinit la cohomologie \`a support compact $H^\bullet_c(\GG(\Q),M)$
comme la cohomologie du c\^one $[\rg(\GG(\Q),M)\to\rg(\BB(\Q),M)]$.  
On a une suite exacte longue
\begin{equation}\phantomsection\label{cohcom}
\cdots\to H^{i-1}(B,M)\to H^i_c(G,M)\to H^i(G,M)\to H^i(B,M)\to H^{i+1}_c(G,M)\to \cdots
\end{equation}
La cohomologie \`a support compact est calcul\'ee
par le complexe\footnote{${\cal C}$ d\'esigne les fonctions continues, 
mais comme $\GG(\Q)$ est discret,
toutes les fonctions sont continues.}
(dans lequel on note simplement $B$ et $G$ les groupes $\BB(\Q)$ et $\GG(\Q)$)
$$M\to {\cal C}(G,M)\oplus M\to {\cal C}(G\times G,M)\oplus {\cal C}(B,M)\to
{\cal C}(G\times G\times G,M)\oplus {\cal C}(B\times B,M)\to\cdots\,,$$
o\`u les fl\`eches ${\cal C}(H^i,M)\to {\cal C}(H^{i+1},M)$, pour $H=G,B$, sont
les diff\'erentielles usuelles, les fl\`eches ${\cal C}(G^i,M)\to {\cal C}(B^i,M)$
sont les restrictions, et les autres fl\`eches sont nulles\footnote{On d\'efinit de m\^eme
$H^1_c(\Gamma(1),M)$.}.  
En particulier,
$$H^1_c(\GG(\Q),M)=\frac{\{((c_\sigma)_{\sigma\in \GG(\Q)},c_B),\ c_{\sigma\tau}=\sigma\cdot c_\tau+c_\sigma,\ 
c_\sigma=(\sigma-1)\cdot c_B,\ {\text{si $\sigma\in \BB(\Q)$}}\}}
{\{(((\sigma-1)\cdot a)_{\sigma\in \GG(\Q)},a),\ a\in M\}}.$$
On note $Z^1(\GG(\Q),\BB(\Q),M)$ le module
des $1$-cocycles $(c_\sigma)$ sur $\GG(\Q)$, 
\`a valeurs dans $M$, qui sont identiquement nuls
sur $\BB(\Q)$.
On dispose d'une application naturelle
$Z^1(\GG(\Q),\BB(\Q),M){\to} H^1_c(\GG(\Q),M)$ envoyant
$(c_\sigma)$ sur la classe de $((c_\sigma),0)$.
\begin{lemm}\phantomsection\label{ES2}
Cette application induit un isomorphisme naturel
$$Z^1(\GG(\Q),\BB(\Q),M)\overset{\sim}{\to} H^1_c(\GG(\Q),M).$$
\end{lemm}
\begin{proof}
Cela r\'esulte de ce que
$$\phantom{XXXXX}((c_\sigma),c_B)=((c_\sigma-(\sigma-1)\cdot c_B),0)+
(((\sigma-1)\cdot c_B),c_B).\phantom{XXXXX}\qedhere$$
\end{proof}

\Subsubsection{Lien entre cohomologie et cohomologie \`a support compact}
\begin{theo}\phantomsection\label{cococo}
On a $H^i_c(\GG,{\cal C})=0$ si $i\neq 1$, et
une suite exacte de $\GG(\A)$-modules
$$0\to {\cal C}(\wZ^\dual)\to{\rm Ind}_{\BB(\A)}^{\GG(\A)}{\cal C}(\LL(\wZ))\to
H^1_c(\GG,{\cal C})\to H^1(\GG,{\cal C})\to 0$$
o\`u $\wZ^\dual=\A^\dual/\R_+^\dual\Q^\dual$ sur lequel $\GG(\A)$ agit \`a travers le d\'eterminant,
et $\BB(\A)$ \`a travers $\LL(\A)$ sur $\LL(\wZ)$.
\end{theo}
\begin{proof}
On utilise la suite exacte longue~(\ref{cohcom}) pour $M={\cal C}(\GG(\A))$.
Comme $H^0(\GG,{\cal C})\to H^0(\BB(\Q),{\cal C}(\GG(\A))))$
est injective, on a $H^0_c(\GG,{\cal C})=0$.
Comme $H^i(\GG,{\cal C})=0$ pour $i\geq 2$ et $H^0(\GG,{\cal C})={\cal C}(\wZ^\dual)$ 
(prop.\,\ref{coco6}), il suffit de prouver que $H^i(\BB(\Q),{\cal C}(\GG(\A)))=0$,
pour $i\geq 1$, et d'identifier $H^0(\BB(\Q),{\cal C}(\GG(\A)))$, ce qui fait l'objet
de la prop.\,\ref{indu5}
\end{proof}

%\begin{rema}
%Comme les vecteurs localement alg\'ebriques de $H^1_c$ et $H^1$ sont les vecteurs classiques
%(d'apr\`es Emerton),
%la fl\`eche $(H^1_c)^{\rm alg}\to (H^1)^{\rm alg}$ n'est pas surjective (les s\'eries
%principales fournies par les s\'eries d'Eisenstein disparaissent).
%Le terme ${\rm Ind}_{B(\A)}^{\GG(\A)}{\cal C}(\wZ^\dual\times\wZ^\dual)$ interpole
%ces s\'eries principales (probablement du genre $\chi_1\otimes B(\chi_1,\chi_2)$
%o\`u le premier $\chi_1$ est l'action de Galois et $B(\chi_1,\chi_2)$ est l'induite
%convenablement normalis\'ee.
%
%Il y a probablement dans $H^1(G,{\cal C})$ un second morceau qui interpole
%les $\chi_1\otimes B(\chi_2,\chi_1)$ et \c{c}a fabrique une extension non triviale
%$B(\chi_1,\chi_2)-B(\chi_2,\chi_1)$ (en $\ell\neq p$ ces repr\'esentations co\"{\i}ncident
%et l'extension n'a lieu qu'en $p$ et il faut se m\'efier du cas Steinberg...).  
%
%Je me demande si on peut exhiber ce second morceau? Et d\'ecrire cette extension?
%\c{C}a ressemble a ce que racontait Lecouturier avec sa bi-extension.
%\end{rema} 
%
%\begin{rema} 
%Si $\chi_1,\chi_2$ sont des caract\`eres lisses de $\A^{]\infty[,\dual}$, soit
%$I(\chi_1,\chi_2):={\rm Ind}_{\BB(\A)}^{\GG(\A)}(\chi_2\otimes\chi_1\delta_\A^{-1})$ (induite lisse).
%On a $I(\chi_1,\chi_2)\cong I(\chi_2,\chi_1)$ sauf si $\chi_2=\chi_1\delta_\A^{-1}$
%o\`u $\chi_2$ est une sous-repr\'esentation de $I(\chi_1,\chi_2)$ et un quotient 
%de $I(\chi_2,\chi_1)$, et si $\chi_1=\chi_2\delta_\A^{-1}$ o\`u c'est l'inverse.
%\end{rema}

\Subsection{Dualit\'e entre cohomologie et cohomologie \`a support compact}
\subsubsection{Descente de $\GG(\Q)$ \`a $\Gamma$ pour la cohomologie \`a support compact}
Le lemme de Shapiro n'a, a priori, pas de raison
d'\^etre vrai pour la cohomologie \`a support compact (on a quand m\^eme 
une fl\`eche naturelle $H^i_c(\GG(\Q),{\rm Ind}_\Gamma^{\GG(\Q)}M)\to H^i_c(\Gamma,M)$
obtenue en \'evaluant les fonctions $\phi:\GG(\Q)\to M$ en $1$).
Mais on a le r\'esultat suivant:
\begin{prop}\phantomsection\label{cup7}
Si $X={\cal C}, {\cal C}^{(p)},{\rm LA},{\rm LP},{\rm LC}$,
l'application naturelle induit, pour tout $i\geq 0$, un isomorphisme
$$H^i_c(\GG(\Q),X(\GG(\A),L))\overset{\sim}{\to} H^i_c(\Gamma,X(\GG(\wZ),L))$$
\end{prop}
\begin{proof}
On a un diagramme commutatif \`a lignes exactes, o\`u 
$$B=\BB(\Q),\hskip2mm G=\GG(\Q),\hskip2mm U=\BB(\Q)\cap\Gamma, \hskip2mm
X_\A=X(\GG(\A),L),\hskip2mm X_\wZ=X(\GG(\wZ),L)$$
et les deux isomorphismes verticaux r\'esultent du lemme de Shapiro:
$$\xymatrix@R=.4cm@C=.5cm{
H^{i-1}(G,X_{\A})\ar[d]^\wr\ar[r]& H^{i-1}(B,X_{\A})\ar[d]\ar[r]
&H^i_c(G,X_{\A})\ar[d]\ar[r] 
&H^i(G,X_{\A})\ar[d]^\wr\ar[r]& H^i(B,X_{\A})\ar[d]\\
H^{i-1}(\Gamma,X_{\wZ})\ar[r]& H^{i-1}(U,X_{\wZ})\ar[r]
&H^i_c(\Gamma,X_{\wZ})\ar[r] 
&H^i(\Gamma,X_{\wZ})\ar[r]& H^i(U,X_{\wZ})}$$
Pour conclure, gr\^ace au lemme des 5, il suffit de v\'erifier
que les fl\`eches $H^j(B,X_\A)\to H^j(U,X_{\wZ})$, pour $j=i-1,i$,
sont des isomorphismes.

On a $\GG(\A)=\BB(\Q)\cdot(\GG(\wZ)\GG(\R)_+)$ et $\BB(\Q)\cap(\GG(\wZ)\GG(\R)_+)=U$.
Il en r\'esulte que $X(\GG(\A))\cong{\rm Ind}_U^B X(\GG(\wZ))$,
et on conclut
en utilisant le lemme de Shapiro.
\end{proof}

\Subsubsection{Dualit\'e entre $H^0$ et $H^2_c$}
On reprend les notations du \no\ref{gamma}.
%Soient $\Gamma={\rm SL}_2(\Z)$ et $\oGamma={\rm SL}_2(\Z)/\{\pm I\}$, 
%o\`u $I=\matrice{1}{0}{0}{1}$.
%Si $M$ est un $\Gamma$-module, on a $H^i(\Gamma,M)=H^i(\oGamma,M^{\{\pm I\}})$ \`a $2$-torsion pr\`es.
Le groupe $\oGamma$ est engendr\'e par $S=\matrice{0}{1}{-1}{0}$ et $U=\matrice{0}{1}{-1}{1}$,
les seules relations \'etant $S^2=1$ et $U^3=1$.  De plus, si $\overline B\subset\oGamma$ est l'image
de $B=\matrice{1}{\Z}{0}{1}$, alors $SU$ est un g\'en\'erateur de $\overline B$
(c'est l'image de $\matrice{1}{-1}{0}{1}$).

Si $M$ est un $\Z_p[\oGamma]$-module, on note $M^\dual$ son dual (suivant les cas, cela
peut \^etre le dual de Pontryagin, le $\Z_p$-dual, ou le $\Q_p$-dual (ou $L$-dual) topologique).

\begin{prop}\phantomsection\label{h2c}
{\rm (i)}  On a une identification naturelle 
$$H^2_c(\oGamma,M^\dual)=M^\dual/(U-1,S-1)$$

{\rm (ii)} $H^2_c(\oGamma,\Lambda)=\Lambda$.

{\rm (iii)}
 Le cup produit $H^2_c(\oGamma,M^\dual)\times H^0(\oGamma,M)\to H^2_c(\oGamma,\Lambda)=\Lambda$
induit un isomorphisme
$$H^0(\oGamma,M)\cong H^2_c(\oGamma,M^\dual)^\dual$$
\end{prop}
\begin{proof}
Comme $H^2(\oGamma,M^\dual)=0$
(puisque $\oGamma$ contient
un groupe libre d'indice $6$), on a 
$H^2_c(\oGamma,M^\dual)=H^1(\overline{B},M^\dual)/
H^1(\oGamma,M^\dual)$.
  Comme $\overline{B}$ est engendr\'e par $SU$, on a
$H^1(\overline{B},M^\dual)=M^\dual/(SU-1)$, o\`u un cocycle $\gamma\mapsto c_\gamma$ est
envoy\'e sur l'image de $c_{SU}$.  

Par ailleurs, comme on l'a vu plus haut un $1$-cocyle $\tau\mapsto c_\tau$ sur $\oGamma$ est
cohomologue \`a un $1$-cocycle v\'erifiant $c_S=0$ et $c_U\in (U-1)M^\dual$ arbitraire.
On a alors $c_{SU}=Sc_U$ et donc
$H^2_c(\oGamma,M^\dual)=M^\dual/(SU-1,S(U-1))=M^\dual/(S-1,U-1)$.

Ceci prouve le (i).  Le (ii) s'en d\'eduit imm\'ediatement, et le (iii) r\'esulte
de ce que $H^0(\oGamma,M)=M^{S=1,U=1}$.
\end{proof}

\begin{rema}\phantomsection\label{explicit}
On peut expliciter l'isomorphisme 
$$H^2_c(\oGamma,M^\dual)=H^1(\overline{B},M^\dual)/
H^1(\oGamma,M^\dual)$$ utilis\'e dans la preuve.
Une classe $c\in H^2_c(\oGamma,M^\dual)$ est repr\'esent\'ee par 
$((c_{\sigma,\tau})_{\sigma,\tau},(b_\sigma)_{\sigma})$
o\`u $(\sigma,\tau)\mapsto c_{\sigma,\tau}$ est un $2$-cocycle sur $\oGamma$,
et $b_\sigma:\overline{B}\to M^\dual$ v\'erifie 
$c_{\sigma,\tau}=b_{\sigma\tau}-\sigma\cdot b_\tau-b_\sigma$, pour tous $\sigma,\tau\in\overline{B}$.
Comme $H^2(\oGamma,M^\dual)=0$, on peut trivialiser le $2$-cocycle (i.e.~l'\'ecrire
comme le bord de $\sigma\mapsto c_\sigma$ (bien d\'etermin\'e \`a addition pr\`es
d'un $1$-cocycle sur $\Gamma$)), et obtenir un repr\'esentant
de la forme $((0)_{\sigma,\tau},(b'_\sigma)_{\sigma})$, o\`u 
$\sigma\mapsto b'_\sigma=b_\sigma-c_\sigma$
est un $1$-cocycle sur $\overline{B}$, unique \`a addition pr\`es
de la restriction \`a $\overline{B}$ d'un $1$-cocycle sur $\oGamma$.

Dans le cas o\`u $M=\Lambda$ (et donc $M^\dual=\Lambda$), un $1$-cocycle sur $\oGamma$
est identiquement nul puisque $\oGamma$ est engendr\'e par $S$ et $U$ qui sont de torsion.
Il s'ensuit que le $1$-cocycle $(b'_\sigma)_{\sigma}$ ci-dessus est uniquement
d\'etermin\'e et l'isomorphisme $H^2_c(\oGamma,\Lambda)\cong\Lambda$ est celui envoyant
$c$ sur $b'_{SU}$.
\end{rema}

\subsubsection{Dualit\'e entre $H^1$ et $H^1_c$}\label{cup3}
Le groupe $H^1_c(\oGamma,M^\dual)$ est le groupe des $1$-cocycles $\sigma\mapsto c^\dual_\sigma$
sur $\oGamma$, \`a valeurs dans $M^\dual$, qui sont identiquement nuls sur $\overline B$.
Un tel cocycle est enti\`erement d\'etermin\'e par $c^\dual_U$ et $c^\dual_S$, et on a
$(1+S)c^\dual_S=0$ et $(1+U+U^2)c^\dual_U=0$ (\`a cause des relations $S^2=1$ et $U^3=1$)
et la relation $c^\dual_{SU}=0$ impose en plus que $Sc^\dual_U+c^\dual_S=0$, ou encore
$Sc^\dual_U-Sc^\dual_S=0$, et donc $c^\dual_U=c^\dual_S$.
Autrement dit, on a une identification
\begin{equation}\phantomsection\label{h1.2}
H^1_c(\oGamma,M^\dual)=(M^\dual)^{1+U+U^2=0}\cap (M^\dual)^{1+S=0}
\end{equation}

%Maintenant, un $1$-cocycle $\sigma\mapsto c_\sigma$ sur $\oGamma$, \`a valeurs dans $M$,
%est uniquement d\'etermin\'e par $c_U$ et $c_S$ soumis aux relations
%$(1+S)c_S=0$ et $(1+U+U^2)c_U=0$. La classe de cohomologie de ce cocycle
%ne change pas si on remplace $(c_U,c_S)$ par $(c_U-(U-1)c,c_S-(S-1)c)$, avec
%$c\in M$.  Cela permet, \`a $2$-torsion pr\`es, de supposer que $c_S=0$, et alors
%la classe de cohomologie est d\'etermin\'ee par $c_U$ \`a addition pr\`es de
%$(U-1)c$, avec $c\in M^{S=1}$. On a donc un isomorphisme
Par ailleurs, il r\'esulte de la suite exacte (\ref{gamma1}) que l'on a
\begin{equation}\phantomsection\label{h1.1}
H^1(\oGamma, M)\cong M^{1+U+U^2=0}/(U-1)M^{S=1}
\end{equation}

Notons $\langle\ ,\ \rangle:M^\dual\times M\to \Lambda$, 
o\`u $\Lambda=\Z_p$ ou $\Q_p/\Z_p$ ou $\Q_p,L$..., l'accouplement
naturel. L'accouplement 
$$H^1_c(\oGamma,M^\dual)\times H^1(\oGamma, M)\to H^2_c(\oGamma,\Lambda)=\Lambda$$
induit,
en utilisant les identification pr\'ec\'edentes, un accouplement
$$\cup:\big((M^\dual)^{1+U+U^2=0}\cap (M^\dual)^{1+S=0}\big)\times
\big(M^{1+U+U^2=0}/(U-1)M^{S=1}\big)\to\Lambda$$
\begin{lemm}\phantomsection\label{ES66}
On a $3(c^\dual\cup c)=\langle c^\dual,(1-U^2)c\rangle$.
\end{lemm}
\begin{proof}
Soient $\sigma\mapsto c^\dual_\sigma$ le $1$-cocycle sur $\oGamma$
correspondant \`a $c^\dual$ et $\tau\mapsto c_\tau$ celui correspondant \`a $c$.
En particulier:

$\bullet$ $c^\dual_\sigma=0$ si $\sigma\in\overline{B}$, et $c^\dual_S=c^\dual_U=c^\dual$.

$\bullet$ $c_S=0$ et $c$ est l'image de $c_U$.

Alors 
$(\sigma,\tau)\mapsto c_{\sigma,\tau}:=\langle c^\dual_\sigma, \sigma\cdot c_\tau\rangle$
est un $2$-cocycle sur $\Gamma$ (dont la
restriction \`a $\overline{B}\times\overline{B}$, et m\^eme \`a $\overline{B}\times\oGamma$
est identiquement nulle), et $c^\dual\cup c$ 
est la classe de $((c_{\sigma,\tau})_{\sigma,\tau},(0)_\sigma)$.
Comme $H^2(\oGamma,\Lambda)=0$,
il existe
$\phi:\Gamma\to\Lambda$, unique, telle que l'on ait
$$\langle c^\dual_\sigma, \sigma\cdot c_\tau\rangle=\phi(\sigma\tau)-
\phi(\sigma)-\phi(\tau)$$
et on a $c^\dual\cup c=\phi(SU)$ (cf.~rem.\,\ref{explicit}).

Par exemple, si $c_\tau=(\tau-1)c$, on a $\phi(\sigma)=\langle c^\dual_\sigma, \sigma\cdot c\rangle$
et $c^\dual\cup c=\langle c^\dual_{SU}, SU\cdot c\rangle=0$ (comme il se doit)
puisque $c^\dual_{SU}=0$ vu que $SU\in\overline B$.

Par d\'efinition,
$\phi(SU)=\langle c^\dual_S,Sc_U\rangle+\phi(S)+\phi(U)$.
Notons que $\phi(1)=0$ (appliquer la formule \`a $\sigma=\tau=1$).
Par ailleurs $\phi(1)-2\phi(S)=\langle c^\dual_S,Sc_S\rangle$
(appliquer la formule \`a $\sigma=\tau=S$), et comme $c_S=0$, on obtient $\phi(S)=0$.

Maintenant, si on applique la formule pour $\sigma=\tau=U$ et pour $\sigma=U$, $\tau=U^2$,
on obtient
$$\langle c^\dual_U,Uc_U\rangle=\phi(U^2)-2\phi(U),\quad
 \langle c^\dual_U,Uc_{U^2}\rangle=-\phi(U)-\phi(U^2)$$
Comme $Uc_{U^2}=c_{U^3}-c_U=-c_U$, on en tire
$3\phi(U)=\langle c^\dual_U,(1-U)c_U\rangle,$
et comme $c^\dual_U=c^\dual_S$, on obtient finalement,
$$3(c^\dual\cup c)=\langle c^\dual,(U+2)c\rangle=\langle c^\dual,(1-U^2)c\rangle$$
car l'adjoint de $U^2+U+1$ est $U^{-2}+U^{-1}+1=U+U^2+1$
qui tue $c^\dual$, et $U+2=(1-U^2)+(1+U+U^2)$.
\end{proof}

\begin{prop}\phantomsection\label{ES5}
L'accouplement $\cup$ induit un isomorphisme
 (\`a $6$-torsion pr\`es)
$$H^1_c(\oGamma,M^\dual)\cong H^1(\oGamma,M)^\dual$$
\end{prop}
\begin{proof}
A $3$-torsion pr\`es, on a $M=M^{U=1}\oplus M^{U^2+U+1=0}$ et
$M^\dual=(M^\dual)^{U=1}\oplus (M^\dual)^{U^2+U+1=0}$.
L'orthogonal de $M^{U=1}$ est $(M^\dual)^{U^2+U+1=0}$, et donc
le dual de $M^{U^2+U+1=0}$ pour $\langle\ ,\ \rangle$
s'identifie \`a $(M^\dual)^{U^2+U+1=0}$. Comme $c\mapsto (1-U^2)c$ est un isomorphisme
de $M^{1+U+U^2=0}$ (\`a $3$-torsion pr\`es), il en est de m\^eme pour l'accouplement $\cup$.

Il s'ensuit (cf.~formule~(\ref{h1.1})) que
$H^1(\oGamma,M)^\dual=(M^\dual)^{U^2+U+1=0}\cap ((U-1)M^{S=1})^\perp$ (o\`u $\perp$ d\'esigne
l'orthogonal pour $\cup$).
Or le membre de droite est l'ensemble des
$c^\dual\in (M^\dual)^{U^2+U+1=0}$ v\'erifiant
$\langle c^\dual,(1-U^2)(U-1)c\rangle=0$ pour tout $c\in M^{S=1}$.
Comme $(1-U^2)(U-1)=-(U^2+U+1)-3$ et comme l'adjoint de $U^2+U+1$ est $U^{-2}+U^{-1}+1=U+U^2+1$
qui tue $c^\dual$, le membre de droite est aussi l'ensemble des
$c^\dual\in (M^\dual)^{U^2+U+1=0}$ v\'erifiant
$\langle c^\dual,3c\rangle=0$ pour tout $c\in M^{S=1}$.
Comme cette derni\`ere condition \'equivaut \`a $c^\dual\in (M^\dual)^{S+1=0}$ 
(\`a $2$-torsion pr\`es),
on en d\'eduit le r\'esultat gr\^ace \`a la description~(\ref{h1.2})
de $H^1_c(\oGamma,M^\dual)$.
\end{proof}
\begin{rema}\phantomsection\label{ES6}
{\rm (i)} Si $M$ est un $\Z_p[\Gamma]$-module,
on en d\'eduit un isomorphisme $H^1_c(\Gamma,M^\dual)\cong H^1(\Gamma,M)^\dual$
(\`a $12$-torsion pr\`es).
 
{\rm (ii)} L'isomorphisme dans l'autre sens $H^1(\Gamma,M)\cong H^1_c(\Gamma,M^\dual)^\dual$
n'est pas automatique, m\^eme si $(M^\dual)^\dual=M$: 
par exemple, si $M$ est $L$-espace vectoriel topologique,
il faut que $H^1(\Gamma,M)$ soit s\'epar\'e; si $M$ est un $\O_L$-module sans torsion,
il faut que $H^1(\Gamma,M)$ soit sans torsion. 
\end{rema}

\begin{coro}\phantomsection\label{ES6.4}
Les groupes $H^1_c(\GG,{\cal C})$ et $H^1(\GG,{\rm Mes})$
sont en dualit\'e.
\end{coro}
\begin{proof}
D'apr\`es le cor.\,\ref{coco5} et la prop.\,\ref{cup7},
on a des isomorphismes:
\begin{align*}
H^1_c(\GG,{\cal C})&\cong H^1_c(\Gamma,{\cal C}(\GG(\wZ),L))\\
H^1(\GG,{\rm Mes})&\cong H^1(\Gamma,{\rm Mes}(\GG(\wZ),L))
\end{align*}
La prop.\,\ref{ES5} (cf.~rem.\,\ref{ES6}) implique donc une dualit\'e dans un sens.
Pour en d\'eduire celle dans l'autre sens, il suffit de v\'erifier que
$H^1(\Gamma,{\rm Mes}(\GG(\wZ),\O_L))$ est sans $p$-torsion, ce qui
r\'esulte de ce que $H^0(\Gamma,{\rm Mes}(\GG(\wZ),k_L))=0$ car une mesure invariante par
$\Gamma$ l'est aussi par $\matrice{1}{\Z_p}{0}{1}\subset\matrice{1}{\wZ}{0}{1}$ par continuit\'e
de l'action de $\matrice{1}{\wZ}{0}{1}$,
et donc est nulle (pas de mesure de Haar en $p$-adique).
\end{proof}

\Subsection{Cohomologie \`a support compact et vecteurs localement alg\'ebriques}
\Subsubsection{Cohomologie des fonctions localement alg\'ebriques}

\begin{theo}\phantomsection\label{alge1}
Si $X={\rm LC},{\rm LP}$,
on a 
$$H^i_c(\GG,X)=\begin{cases}0 &{\text{si $i\neq 1,2$,}}\\ X(\wZ^\dual) &{\text{si $i=2$,}}
\end{cases}$$
et on a une suite exacte de $\GG(\A)$-modules:
$$\xymatrix@R=3mm@C=4mm{
0\ar[r]& X(\wZ^\dual)\ar[r]&{\rm Ind}_{\BB(\A)}^{\GG(\A)}X(\LL(\wZ))\ar[r]&
H^1_c(\GG,X)\ar[dl]\\
& & H^1(\GG,X)\ar[r]&
\big({\rm Ind}_{\BB(\A)}^{\GG(\A)}X(\LL(\wZ))\big)^\vee
\ar[r]& X(\wZ^\dual)\ar[r]& 0
}$$
\end{theo}
\begin{proof}
Commen\c{c}ons par calculer $H^2_c$.
On a $H^i_c(\GG,X)=H^i_c(\Gamma,X(\GG(\wZ)))$ d'apr\`es la prop.\,\ref{cup7}.
Par ailleurs, $\GG(\wZ)=\GG'(\wZ)\cdot\matrice{\wZ^\dual}{0}{0}{1}$, et 
donc $H^2_c(\GG,X)\cong H^2_c(\Gamma,\GG'(\wZ))\otimes X(\wZ^\dual)$,
et on est ramen\'e \`a prouver que $H^2_c(\Gamma,X(\GG'(\wZ)))=L$.
Or $X(\GG'(\wZ))$ est une limite inductive de repr\'esentations de $\Gamma$
de dimension finie (de la forme $\oplus_W{\cal C}(\GG(\Z/N))\otimes(W\otimes W^\dual)$,
la somme \'etant sur un ensemble fini de repr\'esentations alg\'ebriques irr\'eductibles
de $\GG'(\Q_p)$).  Il r\'esulte de la prop.\.,\ref{h2c} que
$H^2_c(\Gamma,X(\GG'(\wZ)))$ est le dual de $H^0(\Gamma,X(\GG'(\wZ))^\dual)$.
Mais $H^0(\Gamma,-)=H^0(\GG'(\wZ),-)$ car $\Gamma$ est dense dans $\GG'(\wZ)$.
Utiliser l'action de l'alg\`ebre de Lie permet de montrer que seule la partie
correspondant \`a $W=1$ contribue; i.e.~on obtient le m\^eme r\'esultat pour
$X={\rm LC}$ que pour $X={\rm LP}$.  Maintenant, si $X={\rm LC}$, il est
facile de voir que la seule forme lin\'eaire invariante (\`a multiplication
pr\`es par un scalaire) est la\footnote{
Cette mesure n'est pas born\'ee ce qui explique
qu'il n'y ait pas de vraie mesure de Haar comme nous l'avons affirm\'e dans la
preuve du cor.\,\ref{ES6.4}).} ``mesure de Haar'' $\mu$
d\'efinie par $\mu({\bf 1}_{\gamma K})=\frac{1}{[\GG'(\wZ):K]}$ si $K$ est un sous-groupe
ouvert compact de $\GG'(\wZ)$, et $\gamma\in\GG'(\wZ)$.  On a donc
$H^0(\Gamma, X(\GG'(\wZ))^\dual)\cong L$;
on en d\'eduit le r\'esultat pour~$H^2_c$.

Le reste de l'\'enonc\'e r\'esulte d'une combinaison de
la suite exacte~(\ref{cohcom}) pour $M=X(\GG(\A))$, du th.\,\ref{indu6} qui
donne la valeur de $H^i(B,M)$ pour tout $i$, et de la prop.\,\ref{coco6} 
qui donne celle de $H^i(G,M)$ pour $i\neq 1$.
\end{proof}
\begin{rema}\phantomsection\label{alge2}
Le th.\,\ref{alge1} semble ne pas \^etre compatible avec le th.\,\ref{cococo}
mais
$H^1(\BB(\Q),{\rm LC}(\GG(\A),\O_L))$ est un $L$-espace vectoriel (cela r\'esulte de
la rem.\,\ref{gagm2.1} et de la prop.\,\ref{indu6.0}), et donc
son image dans $H^1(\BB(\Q),{\cal C}(\GG(\A),\O_L))$ est {\it a priori} nulle;
il n'est donc pas absurde que 
$H^1(\BB(\Q),{\cal C}(\GG(\A)))=0$ bien que $H^1(\BB(\Q),{\rm LC}(\GG(\A)))\neq 0$.
\end{rema}

\subsubsection{Vecteurs localement alg\'ebriques}
Soit $W$ une repr\'esentation alg\'ebrique de~$\GG$, que l'on voit comme
une repr\'esentation de $\GG(\A)$ agissant \`a travers $\GG(\Q_p)$.
Tensoriser la suite exacte 
$$0\to {\cal C}(\wZ^\dual)\to{\rm Ind}_{\BB(\A)}^{\GG(\A)}{\cal C}(\LL(\wZ))\to
H^1_c(\GG,{\cal C})\to H^1(\GG,{\cal C})\to 0$$
par $W$, fournit encore une suite exacte.  
Soit $K$ un sous-groupe ouvert de $\GG(\wZ)$.  
D'apr\`es le lemme~\ref{coco200}, on a
$H^i(K,{\cal C}(\wZ^\dual)\otimes W)$ si $i=1,2$.
Il s'ensuit que
$$H^1(K,\big({\rm Ind}_{\BB(\A)}^{\GG(\A)}{\cal C}(\LL(\wZ))\big)\otimes W)\overset{\sim}{\to}
H^1\big(K,\big(\big({\rm Ind}_{\BB(\A)}^{\GG(\A)}{\cal C}(\LL(\wZ))\big)/{\cal C}(\wZ^\dual)\big)\otimes W)$$
et que l'on a une suite exacte
\begin{equation}\phantomsection\label{suite}
\xymatrix@R=4mm@C=4mm{0\to\big({\cal C}(\wZ^\dual)\otimes W)^K\ar[r]&
\big(\big({\rm Ind}_{\BB(\A)}^{\GG(\A)}{\cal C}(\LL(\wZ))\big)\otimes W\big)^K\ar[r]&
\big(H^1_c(\GG,{\cal C})\otimes W\big)^K\ar[d]\\
&H^1(K,\big({\rm Ind}_{\BB(\A)}^{\GG(\A)}{\cal C}(\LL(\wZ))\big)\otimes W)
& \big(H^1(\GG,{\cal C})\otimes W\big)^K\ar[l]&}
\end{equation}

\begin{theo}\phantomsection\label{alge10}
$H^1_c(\GG,{\cal C})^{\rm alg}=H^1_c(\GG,{\rm LP})$
\end{theo}
\begin{proof}
Il suffit de prouver que, pour tout sous-groupe ouvert compact $K$ de $\GG(\wZ)$,
et toute repr\'esentation alg\'ebrique $W$ de $\GG(\Q_p)$, on a
$$H^0(K,H^1_c(\GG(\Q),{\cal C}(\GG(\A)))\otimes W)=H^1_c(\GG(\Q),H^0(K,{\cal C}(\GG(\A))\otimes W))$$
En effet, si on fixe $W$ et que l'on prend la limitive inductive sur $K$, le membre de gauche
calcule la $W^\dual$-composante de $H^1_c(\GG,{\cal C})^{\rm alg}$ et le membre de droite
celle de $H^1_c(\GG,{\rm LP})$: d'apr\`es la prop.\,\ref{cup7},
on a $H^1_c(\GG,{\rm LP})=H^1_c(\Gamma,{\rm LP}(\GG(\wZ)))$; par ailleurs
\begin{align*}
H^1_c(\Gamma,{\rm LP}(\GG(\wZ)))&=\oplus_W
H^1_c(\Gamma,\varinjlim\nolimits_K H^0(K,{\cal C}(\GG(\wZ))\otimes W\otimes W^\dual))\\
&=\oplus_W
\varinjlim\nolimits_K H^1_c(\Gamma,H^0(K,{\cal C}(\GG(\wZ))\otimes W))\otimes W^\dual\\
&=\oplus_W
\varinjlim\nolimits_K H^1_c(\GG(\Q),H^0(K,{\cal C}(\GG(\A))\otimes W))\otimes W^\dual
\end{align*}
(le dernier isomorphisme se d\'emontre comme la prop.\,\ref{cup7}).

Si $\Gamma$ est un sous-groupe de $\GG(\Q)\times\GG(\A)$, posons 
$H^i(\Gamma):=H^i(\Gamma,
{\cal C}(\GG(\A))\otimes W)$. On a les annulations suivantes:

\quad $\bullet$  
$H^i(\BB(\Q))=0$ si $i\geq 1$, d'apr\`es la prop.\,\ref{indu5}.

\quad $\bullet$ $H^i(K)=0$ si $i\geq 1$, car ${\cal C}(\GG(\A))\otimes W=
\prod_{g\in \GG(\A)/K}({\cal C}(gK)\otimes W)$ et chaque
${\cal C}(gK)\otimes W$ est une induite de $\{1\}$ \`a $K$ de $W$.

\quad $\bullet$ 
$H^i(K,H^0(\GG(\Q)))=0$ si $i=1,2$,
d'apr\`es le lemme~\ref{coco200}.

On a alors le diagramme commutatif suivant \`a colonnes exactes
$$\xymatrix@R=4mm@C=4mm{0\ar[d]&0\ar[d]&0\ar[d]\\
H^0(\GG(\Q),H^0(K))\ar@{=}[r]\ar[d]&H^0(\GG(\Q)\times K)\ar@{=}[r]\ar[d]
&H^0(K,H^0(\GG(\Q)))\ar[d]\\
H^0(\BB(\Q),H^0(K))\ar@{=}[r]\ar[d]&H^0(\BB(\Q)\times K)\ar@{=}[r]\ar[d]
&H^0(K,H^0(\BB(\Q)))\ar[d]\\
H^1_c(\GG(\Q),H^0(K))\ar[r]\ar[d]&H^1_c(\GG(\Q)\times K)\ar[r]\ar[d]
&H^0(K,H^1_c(\GG(\Q)))\ar[d]\\
H^1(\GG(\Q),H^0(K))\ar[r]^-{\sim}\ar[d]&H^1(\GG(\Q)\times K)\ar[r]^-{\sim}\ar[d]
&H^0(K,H^1(\GG(\Q)))\ar[d]\\
H^1(\BB(\Q),H^0(K))\ar[r]^-{\sim}&H^1(\BB(\Q)\times K)
&H^1(K,H^0(\BB(\Q)))\ar[l]_-{\sim}
}$$
La premi\`ere colonne est la suite (\ref{cohcom}) pour $M={\cal C}(\GG(\A))\otimes W$,
la seconde est l'analogue de la suite exacte pr\'ec\'edente pour $B=\BB(\Q)\times K$
et $G=\GG(\Q)\times K$, et la troisi\`eme est la suite exacte~(\ref{suite}).  Les fl\`eches
des 4-i\`eme et 5-i\`eme lignes proviennent des applications d'inflation et restriction;
ce sont des isomorphismes gr\^ace aux r\'esultats d'annulation rappel\'es ci-dessus.

La commutativit\'e du diagramme est imm\'ediate sauf celle du carr\'e en bas \`a droite.
Pour la v\'erifier, partons d'un $1$-cocycle $(\gamma,k)\mapsto \phi_{\gamma,k}$ 
sur $\GG(\Q)\times K$. Comme $H^1(\BB(\Q))=0$,
on peut supposer $\phi_{\beta,1}=0$ si $\beta\in\BB(\Q)$. La relation de cocycle donne
alors $\phi_{\beta,k}=\beta\phi_{1,k}+\phi_{\beta,1}=\beta\phi_{1,k}$ 
et $\phi_{\beta,k}=k\cdot\phi_{\beta,1}+\phi_{1,k}=
\phi_{1,k}$. Il s'ensuit que $\phi_{1,k}\in H^0(\BB(\Q))$, et $k\mapsto \phi_{1,k}$
est le cocycle repr\'esentant l'image de $(\gamma,k)\mapsto \phi_{\gamma,k}$ dans $H^1(K,H^0(\BB))$.
Il est alors clair que son inflation \`a $\BB(\Q)\times K$ est la restriction
de $(\gamma,k)\mapsto \phi_{\gamma,k}$ \`a $\BB(\Q)\times K$, ce qui prouve la commutativit\'e.

Le lemme des 5 permet d'en d\'eduire que les fl\`eches de la 3-i\`eme ligne sont des isomorphismes,
ce qui permet de conclure.
\end{proof}

\section{Le cas $\GG$ g\'en\'eral}

\subsection{Descente \`a un sous-groupe arithm\'etique}
On fixe un sous-groupe compact maximal $\GG(\wZ)$ de $\GG(\A^{]\infty[})$ comme 
dans le \S\,\ref{fma3}.
Soient $g_1,\dots,g_r$ des repr\'esentants des classes $\GG(\Q)\backslash \GG(\A)/\GG(\R)_+\GG(\wZ)$,
et soit 
$$\Gamma_i:=\GG(\Q)\cap g_{i}\GG(\wZ)\GG(\R)_+g_i^{-1}$$
Si $\alpha\in\Gamma_i$, alors $g_i^{-1}\alpha g_i=\alpha_i^{]\infty[}\alpha_{\infty,i}$
avec $\alpha_i^{]\infty[}\in\GG(\wZ)$ et $\alpha_{\infty,i}\in\GG(\R)_+$ et
$\alpha\mapsto \alpha_i^{]\infty[}$, $\alpha\mapsto\alpha_{\infty,i}$ sont des morphismes de groupes.

\begin{prop}\phantomsection\label{gener1}
Soit $X$ une classe de fonctions telle que $\phi\in X(\GG(\A))$ si et seulement si
$\kappa\mapsto \phi(x\kappa)\in X(\GG(\wZ))$ pour tout $x\in X(\GG(\A))$.
On a un isomorphisme de $\GG(\Q)$-modules
$$X(\GG(\A))\cong \oplus_i {\rm Ind}_{\Gamma_i}^{\GG(\Q)}(X(\GG(\wZ)))$$
\end{prop}
\begin{proof}
Il suffit de prouver que, pour $i=1,\dots,r$, 
on a un isomorphisme de $\GG(\Q)\times\GG(\wZ)$-modules
$$ X(\GG(\Q)g_i\GG(\wZ)\GG(\R)_+)\cong{\rm Ind}_{\Gamma_i}^{\GG(\Q)}( X(\GG(\wZ)))$$
o\`u $\Gamma_i$ agit sur $ X(\GG(\wZ))$ par $(\alpha*_i\phi)(x)=
\phi((\alpha_i^{]\infty[})^{-1}x)$.

Un \'el\'ement de $\GG(\Q)g_i\GG(\wZ)\GG(\R)_+$ s'\'ecrit, sous la forme
$\gamma^{-1}g_i\kappa x_\infty$, avec $\gamma\in\GG(\Q)$, $\kappa\in \GG(\wZ)$ et $x_\infty\in\GG(\R)_+$
et cette \'ecriture est unique \`a changements simultan\'es $\gamma\mapsto\alpha\gamma$,
$\kappa\mapsto \alpha_i^{]\infty[}\kappa$ et $x_\infty\mapsto \alpha_{\infty,i}x_\infty$,
o\`u $\alpha\in\Gamma_i$.
Si $\phi\in  X(\GG(\Q)g_i\GG(\wZ)\GG(\R)_+)$, on d\'efinit $\phi_\gamma\in X(\GG(\wZ))$
par $\phi_\gamma(\kappa):=\phi(\gamma^{-1}g_i\kappa x_\infty)$ (et le r\'esultat ne d\'epend pas
de $x_\infty$ pour les raisons habituelles).
On a alors 
\begin{align*}
\phi_{\alpha\gamma}(x)&=\phi(\gamma^{-1}\alpha^{-1}g_i\kappa x_\infty)=
\phi(\gamma^{-1}g_i(g_i^{-1}\alpha^{-1}g_i)\kappa x_\infty)\\
&=\phi_\gamma((\alpha_i^{]\infty[})^{-1}\kappa)=\alpha*_i\phi(\kappa)
\end{align*}
ce qui fournit une fl\`eche $\GG(\Q)$-\'equivariante du membre de gauche dans le membre de droite.
R\'eciproquement, si $(\phi_\gamma)_{\gamma\in\GG(\Q)}$, on d\'efinit
$\phi(\gamma^{-1}g_i\kappa x_\infty):=\phi_\gamma(\kappa)$, et le r\'esultat est ind\'ependant du
choix de $\gamma$ car $\phi_{\alpha\gamma}=\alpha*_i\phi$, si $\alpha\in\Gamma_i$.
\end{proof}

Il en r\'esulte, via le lemme de Shapiro, que:

\begin{coro}\phantomsection\label{gener2}
 Si $r\geq 0$ et si $K^{]p[}$ est un sous-groupe ouvert de $\GG(\wZ^{]p[})$,
on a des isomorphismes:
\begin{align*}
H^r(\GG(\Q), X(\GG(\A)))&\cong \oplus_i H^r(\Gamma_i, X(\GG(\wZ)))\\
H^r(\GG(\Q), X(\GG(\A)/K^{]p[}))&\cong \oplus_i H^r(\Gamma_i, X(\GG(\wZ)/K^{]p[}))
\end{align*}
\end{coro}

\begin{rema}\phantomsection\label{gener3}
Les r\'esultats de la prop.\,\ref{gener1} et du cor.\,\ref{gener2} s'appliquent \`a
$X={\cal C}$, ${\cal C}^{(p)}$, ${\rm LA}$, ${\rm LP}$, ${\rm LC}$, etc.
\end{rema}

\begin{coro}
On a
$H^r(\GG(\Q),{\cal C}(\GG(\A)))=L\otimes_{\O_L}H^r(\GG(\Q),{\cal C}(\GG(\A),\O_L))$.
\end{coro}
\begin{proof}
On utilise le cor.~\ref{gener2} pour desscendre de ${\cal C}(\GG(\A))$ \`a ${\cal C}(\GG(\wZ))$
et la compacit\'e de ${\cal C}(\GG(\wZ))$ qui implique que
${\cal C}(\GG(\wZ))=L\otimes_{\O_L}{\cal C}(\GG(\wZ),\O_L)$.
\end{proof} 

\Subsection{La cohomologie de $\Gamma_i$}
\subsubsection{Admissibilit\'e}\label{COCO16}
Soit $K_\infty$ un sous-groupe compact maximal de $\GG(\R)_+$.
L'espace sym\'etrique $\GG(\R)/K_\infty$ est contractile.
Notons $Y(\Gamma_i)$ le quotient $\Gamma_i\backslash \GG(\R)/K_\infty$
(ce quotient peut \^etre une orbifold
si $\Gamma_i$ est trop gros, comme dans le cas $\Gamma={\bf SL}_2(\Z)$).

Si $M$ est un $\Gamma_i$-module, on transforme $M$ 
en un syst\`eme local $\underline M$ sur $Y(\Gamma_i)$ en consid\'erant le quotient
$\Gamma_i\backslash((\GG(\R)/K_\infty)\times M)$).
On a alors $H^r(\Gamma_i,M)=H^r(Y(\Gamma_i),\underline{M})$.
Ceci permet, en triangulant $Y(\Gamma_i)$, de repr\'esenter la cohomologie
de $\Gamma_i$ \`a valeur dans $M$ comme celle d'un complexe 
\begin{equation}\phantomsection\label{rg}
\rg(\Gamma_i,M):=M^{J_0}\to M^{J_1}\to M^{J_2}\to \cdots\to M^{J_d}
\end{equation}
o\`u les $J_i$ sont des ensembles finis et $d$ est la dimension de $Y(\Gamma_i)$.
Notons que les $J_i$ et les applications de connexion ne d\'ependent que de la triangulation
et pas de $M$ (comme dans le lemme~\ref{coco2}).

\begin{prop}\phantomsection\label{gener4}
Si $K^{]p[}$ est un sous-groupe ouvert de $\GG(\wZ^{]p[})$, et si $\Lambda=L,\O_L,\O_L/p^n$,
les $H^r(\GG(\Q),{\cal C}(\GG(\A)/K^{]p[},\Lambda))$ sont des $\Lambda$-repr\'esentations
admissibles de $\GG(\Q_p)$.
\end{prop}
\begin{proof}
Gr\^ace au cor.\,\ref{gener2}, on se ram\`ene \`a prouver le m\^eme \'enonc\'e
pour $H^r(\Gamma_i,{\cal C}(\GG(\wZ)/K^{]p[},\Lambda))$, et le r\'esultat est une
cons\'equence de ce que ${\cal C}(\GG(\wZ)/K^{]p[},\Lambda)$ est admissible comme
repr\'esentation de $\GG(\Z_p)$, et donc le complexe $\rg(\Gamma_i,M)$ est un complexe
de repr\'esentations admissibles de $\GG(\Z_p)$, et donc sa cohomologie est admissible comme
repr\'esentation de $\GG(\Z_p)$ et donc aussi, par d\'efinition, de $\GG(\Q_p)$.
\end{proof}

\subsubsection{Lien avec la cohomologie compl\'et\'ee d'Emerton}
Si $n\geq 1$, on a 
\begin{align*}
{\cal C}(\GG(\wZ),\O_L/p^n)&=\varinjlim\nolimits_K{\cal C}(\GG(\wZ)/K,\O_L/p^n)\\
{\cal C}(\GG(\wZ)/K^{]p[},\O_L/p^n)&=\varinjlim\nolimits_K{\cal C}(\GG(\wZ)/K^{]p[}K_p,\O_L/p^n)
\end{align*}
o\`u $K$ (resp.~$K_p$) d\'ecrit les sous-groupes ouverts de $\GG(\wZ)$ (resp.~$\GG(\Z_p)$).

On d\'efinit la cohomologie compl\'et\'ee d'Emerton par:
\begin{align*}
H^r_{\rm Em}&:=\oplus_i \varprojlim\nolimits_n  H^r(\Gamma_i,{\cal C}(\GG(\wZ),\O_L/p^n))\\
H^r_{\rm Em}(K^{]p[})&:=\oplus_i\varprojlim\nolimits_n H^r(\Gamma_i,{\cal C}(\GG(\wZ)/K^{]p[},\O_L/p^n))
\end{align*}
\begin{prop}\phantomsection\label{gener5}
On a un isomorphisme
$$H^r(\GG(\Q),{\cal C}(\GG(\A)/K^{]p[},\O_L))\overset{\sim}{\to} H^r_{\rm Em}(K^{]p[})$$
\end{prop}
\begin{proof}
On a une suite exacte
\begin{align*}
\xymatrix@R=4mm@C=-10mm{
0\to\oplus_i{\rm R}^1\varprojlim\nolimits_n  H^{r-1}(\Gamma_i,{\cal C}(\GG(\wZ)/K^{]p[},\O_L/p^n))\ar[d]\\
H^r(\GG(\Q),{\cal C}(\GG(\A)/K^{]p[},\O_L))\ar[r]& H^r_{\rm Em}(K^{]p[})\to 0}
\end{align*}
et il suffit de prouver que
 ${\rm R}^1\varprojlim\nolimits_n  H^{r-1}(\Gamma_i,{\cal C}(\GG(\wZ)/K^{]p[},\O_L/p^n))=0$.

Or
$H^{r-1}(\Gamma_i,{\cal C}(\GG(\wZ)/K^{]p[},\O_L/p^n))$ est une repr\'esentation
admissible de $\GG(\Z_p)$, et donc toute suite d\'ecroissante de sous-repr\'esentations
de $\GG(\Z_p)$ est stationnaire. Il s'ensuit que le syst\`eme projectif ci-dessus v\'erifie
le crit\`ere de Mittag-Leffler, ce qui permet de conclure.
\end{proof}

\begin{ques}\phantomsection\label{gener6}
A-t-on $H^r(\GG(\Q),{\cal C}(\GG(\A),\O_L))\overset{\sim}{\to} H^r_{\rm Em}$?
Comme ci-dessus, on a une suite exacte
$$
0\to\oplus_i{\rm R}^1\varprojlim\nolimits_n  H^{r-1}(\Gamma_i,{\cal C}(\GG(\wZ),\O_L/p^n))
\to H^r(\GG(\Q),{\cal C}(\GG(\A),\O_L))\to H^r_{\rm Em}\to 0$$
mais
l'argument d'admissibilit\'e ci-dessus 
fait d\'efaut pour prouver que ${\rm R}^i\varprojlim=0$.
\end{ques}

\Subsection{Vecteurs $\GG(\wZp)$-lisses}\label{COCO5}

On a un diagramme
$$\xymatrix@R=4mm@C=5mm{
H^r(\GG(\Q),{\cal C}^{(p)}(\GG(\A)))\ar[d]^-{\wr}
& H^r(\GG(\Q),{\cal C}(\GG(\A)))^{\GG(\wZp)-{\rm lisses}}\ar[d]^-{\wr} \\
\oplus_i H^r(\Gamma_i,{\cal C}^{(p)}(\GG(\wZ)))\ar[r]
& \oplus_i H^r(\Gamma_i,{\cal C}(\GG(\wZ)))^{\GG(\wZp)-{\rm lisses}}
}$$
qui montre que la fl\`eche naturelle
$H^r(\GG(\Q),{\cal C}^{(p)}(\GG(\A)))\to H^r(\GG(\Q),{\cal C}(\GG(\A)))$ atterrit
dans les vecteurs $\GG(\wZp){\text{-lisses}}$.
Une question naturelle est de savoir si on obtient de la sorte un isomorphisme:
c'est le cas si $\GG={\bf GL}_2$ d'apr\`es la prop.\,\ref{coco14} (cf.~rem.\,\ref{cp1}).
Le cas o\`u $\GG$ est la restriction
des scalaires de $F$ \`a $\Q$ de ${\mathbb G}_m$ montre que cette question innocente
est plus subtile qu'elle n'en a l'air: la fl\`eche ci-dessus est un isomorphisme si
et seulement si la conjecture de Leopoldt est vraie pour $F$ (rem.\,\ref{coco39}). 
Dans le cas g\'en\'eral,
on a juste une suite spectrale reliant les deux termes (prop.\,\ref{gener7}).

\subsubsection{Le cas de ${\bf GL}_1$ sur un corps de nombres}\label{COCO13}
Dans ce ${\rm n}^{\rm o}$, on part d'un corps de nombres $F$, et on prend pour $\GG$ le
groupe alg\'ebrique ${\rm Res}_{F/\Q}{\bf GL}_1$.
On a donc $\GG(\Lambda)=(F\otimes_\Q\Lambda)^\dual$ pour toute $\Q$-alg\`ebre $\Lambda$.
En particulier,
\begin{align*}
\GG(\Q)=F^\dual,\quad 
\GG(\R)=(F\otimes \R)^\dual=(\R\dual)^{r_1}\times(\C^\dual)^{r_2},\\
\GG(\Q_p)=(F\otimes\Q_p)^\dual=\prod\nolimits_{v\mid p}F_v^\dual,
\quad \GG(\A)=(F\otimes \A)^\dual
\end{align*}
Notons que $F^\dual\backslash(F\otimes\A)^\dual/(F\otimes\R)^\dual_+\cong G_F^{\rm ab}$
d'apr\`es la th\'eorie du corps de classes.

\begin{prop}\phantomsection\label{coco35}
On a
$$H^i(\GG(\Q),{\cal C}(\GG(\A)))=\begin{cases}
{\cal C}(F^\dual\backslash(F\otimes\A)^\dual/(F\otimes\R)^\dual_+)&{\text{si $i=0$}}\\
0, &{\text{si $i\geq 1$.}}\end{cases}$$
\end{prop}
\begin{proof}
Le cas $i=0$ est imm\'ediat sur les d\'efinitions, compte-tenu du fait qu'une fonction
continue sur $(F\otimes\R)^\dual_+$, \`a valeurs dans $L$, est constante.

Posons $\GG(\wZ)=(\O_F\otimes\wZ)^\dual$.
Soient $\alpha_1,\dots,\alpha_h$ des repr\'esentants du groupe des
classes $\GG(\Q)\backslash\GG(\A)/\GG(\wZ)\GG(\R)_+$.
Alors $\Gamma_i:=\GG(\Q)\cap \alpha_i\GG(\wZ)\GG(\R)_+\alpha_i^{-1}$ est ind\'ependant de
$i$ car on est dans un groupe commutatif, et est \'egal au groupe $U_F$ des
unit\'es totalement positives de $\O_F$.
On d\'eduit du cor.\,\ref{gener2} et de la commutativit\'e de $\GG$, que
$$H^i(\GG(\Q),{\cal C}(\GG(\A)))\cong H^i(U_F,{\cal C}(\GG(\wZ)))^{\oplus h}$$
Maintenant, si $\widehat U_F$ d\'esigne l'adh\'erence de l'image
de $U_F$ dans le groupe profini $\GG(\wZ)$,
on peut d\'ecomposer $\GG(\wZ)$, en tant que $U_F$-ensemble,
sous la forme $\widehat U_F\times X$, o\`u $X$ est un ensemble profini
muni d'une action triviale de $U_F$ (prendre une section ensembliste continue
de la projection de $\GG(\wZ)$ sur $\GG(\wZ)/\widehat U_F$).
On a alors ${\cal C}(\GG(\wZ))\cong{\cal C}(\widehat U_F)\widehat\otimes
{\cal C}(X)$, et on se ram\`ene \`a prouver que
$H^i(U_F,{\cal C}(\widehat U_F))=0$ si $i\geq 1$.

La  prop.~\ref{cheval1} ci-dessous fournit un isomorphisme
naturel $(U_F\otimes \wZ)\overset{\sim}{\to} \widehat U_F$ 
de groupes topologiques.
On conclut en utilisant le fait que
$H^i(\Z^d,{\cal C}(\wZ^d))=0$, si $i\geq 1$ (cf.~preuve de la prop.\,\ref{gagm3}).
\end{proof}

\begin{prop}\phantomsection\label{cheval1}
L'application naturelle
$(U_F\otimes \wZ)\to \widehat U_F$ est un isomorphisme
de groupes topologiques.
\end{prop}
\begin{proof}
Cela r\'esulte d'un th\'eor\`eme de Chevalley~\cite{cheval} selon lequel
tout sous-groupe d'indice fini de $U_F$ est de congruence
(i.e. contient $U_F(N)=U_F\cap {\rm Ker}(\O_F^\dual\to(\O_F/N)^\dual)$ pour un certain $N$).
En effet, $\widehat U_F$ est le compl\'et\'e de $U_F$ par rapport
aux sous-groupes de congruences, tandis que $U_F\otimes \wZ$ est le compl\'et\'e
profini de $U_F$ (i.e.~le compl\'et\'e par rapport \`a tous les sous-groupes
d'indice fini).  Le th\'eor\`eme de Chevalley implique que les deux syst\`emes
de sous-groupes consid\'er\'es sont cofinaux et donc que les deux compl\'et\'es
sont les m\^emes.
\end{proof}

\begin{rema}\phantomsection\label{coco38}
On conjecture (conjecture de Leopoldt) que l'application naturelle
$(U_F\otimes\Z_p)\to (\O_F\otimes\Z_p)^\dual$ est injective.
\end{rema}

\begin{rema}\phantomsection\label{coco39}
Hill~\cite{hill} a montr\'e que
$H^i(\GG(\Q),{\cal C}^{(p)}(\GG(\A)))=0$ pour tout $i\geq 1$ si et seulement si
la conjecture de Leopoldt est vraie pour $(F,p)$.
Il s'ensuit que
$H^i(\GG(\Q),{\cal C}^{(p)}(\GG(\A)))=H^i(\GG(\Q),{\cal C}(\GG(\A)))^{\GG(\wZp){\text{-lisses}}}$,
pour tout $i$,
si et seulement si la conjecture de Leopoldt est vraie pour $(F,p)$.
\end{rema}

\subsubsection{Le cas g\'en\'eral}\label{COCO15}
Si $H$ est un sous-groupe ferm\'e de $\GG(\A)$, on a deux suites spectrales
de Hochschild-Serre
convergeant vers $H^{i+j}(\GG(\Q)\times H,{\cal C}(\GG(\A)))$, \`a savoir
$H^i(\GG(\Q),H^j(H,{\cal C}(\GG(\A))))$ et
$H^i(H,H^j(\GG(\Q),{\cal C}(\GG(\A)))$.
Or, dans la premi\`ere, tous les termes sont nuls sauf ceux
correspondant \`a $j=0$ (repr\'esentation induite).
On a donc $H^{i+j}(\GG(\Q)\times H,{\cal C}(\GG(\A)))=H^{i+j}(\GG(\Q),{\cal C}(\GG(\A)/H))$,
et 
la suite spectrale de terme $E_2^{i,j}=H^i(H,H^j(\GG(\Q),{\cal C}(\GG(\A))))$
converge vers $H^{i+j}(\GG(\Q),{\cal C}(\GG(\A)/H))$.
On en d\'eduit le r\'esultat suivant:
\begin{prop}\phantomsection\label{gener7}
La suite spectrale de terme 
$$E_2^{i,j}:=\varinjlim\nolimits_{K^{]p[}}H^i(K^{]p[},H^j(\GG(\Q),{\cal C}(\GG(\A))))$$
converge vers $H^{i+j}(\GG(\Q),{\cal C}^{(p)}(\GG(\A)))$.
\end{prop}

\begin{rema}
Si $W$ est une repr\'esentation alg\'ebrique irr\'eductible de $\GG$,
on peut utiliser la technique ci-dessus pour ${\cal C}(\GG(\A))\otimes W$ et
$H$ un sous-groupe ouvert compact de $\GG(\wZ)$, pour retrouver la suite spectrale
d'Emerton~\cite[th.\,0.5]{Em06}.
\end{rema}

\Subsection{Vecteurs localement analytiques}
On a un diagramme
$$\xymatrix@R=4mm@C=5mm{
H^r(\GG(\Q),{\rm LA}(\GG(\A)))\ar[d]^-{\wr}
& H^r(\GG(\Q),{\cal C}^{(p)}(\GG(\A)))^{\rm an}\ar[d]^-{\wr} \\
\oplus_i H^r(\Gamma_i,{\rm LA}(\GG(\wZ)))\ar[r]
& \oplus_i H^r(\Gamma_i,{\cal C}^{(p)}(\GG(\wZ)))^{\rm an}
}$$
qui montre que la fl\`eche naturelle
$H^r(\GG(\Q),{\rm LA}(\GG(\A)))\to H^r(\GG(\Q),{\cal C}^{(p)}(\GG(\A)))$
atterrit
dans les vecteurs localement analytiques.

\begin{prop}\phantomsection\label{gener8}
La fl\`eche 
$H^r(\GG(\Q),{\rm LA}(\GG(\A)))\to H^r(\GG(\Q),{\cal C}^{(p)}(\GG(\A)))^{\rm an}$
est un isomorphisme.
\end{prop}
\begin{proof}
Il suffit de le prouver pour les vecteurs fixes par $K^{]p[}$, o\`u $K^{]p[}$
est un sous-groupe ouvert de $\GG(\wZ^{]p[})$, et on est ramen\'e
\`a prouver que la fl\`eche naturelle
$$ H^r(\Gamma_i,{\rm LA}(\GG(\wZ)/K^{]p[}))\to
H^r(\Gamma_i,{\cal C}^{(p)}(\GG(\wZ)/K^{]p[}))^{\rm an}$$
est un isomorphisme. Cela r\'esulte de:

$\bullet$ l'existence du complexe~(\ref{rg}) calculant la cohomologie de $\Gamma_i$,

$\bullet$ l'exactitude du foncteur ``vecteurs localement analytiques'' pour les
repr\'esentations admissibles~\cite{ST4}.
\end{proof}

\Subsection{Cohomologie \`a support compact}\label{COCO11}
%d'une application naturelle (de restriction)
%$H^i(\GG(\Q),X(\GG(\A)))\to H^i(\PP(\Q),X(\GG(\A)))$.
%On d\'efinit
%la {\it cohomologie parabolique}
%$H^i_{\rm par}(\GG(\Q),X(\GG(\A)))$ comme le noyau
%$$H^i_{\rm par}(\GG(\Q),X(\GG(\A)))={\rm Ker}\Big(H^i(\GG(\Q),X(\GG(\A)))\to 
%\bigoplus_{\PP\,{\rm mod}\sim}H^i(\PP(\Q),X(\GG(\A)))\Big)\,,$$
%la somme portant sur les paraboliques \`a conjugaison pr\`es.

Si $\PP$ est un parabolique de $\GG$, 
la restriction fournit une application naturelle
du complexe ${\rm R}\Gamma(\GG(\Q),{\cal C}(\GG(\A)))$ des cochaines sur $\GG(\Q)$
dans le complexe ${\rm R}\Gamma(\PP(\Q),{\cal C}(\GG(\A)))$ des cochaines sur $\PP(\Q)$.
On d\'efinit
la {\it cohomologie \`a support compact}
$H^i_{\rm c}(\GG(\Q),{\cal C}(\GG(\A)))$
comme la cohomologie du c\^one
$$\Big[{\rm R}\Gamma(\GG(\Q),{\cal C}(\GG(\A)))\to 
\bigoplus_{\PP\,{\rm mod}\sim}{\rm R}\Gamma(\PP(\Q),{\cal C}(\GG(\A)))\Big].$$
%Alors $H^i_{\rm par}(\GG(\Q),{\cal C}(\GG(\A)))$ est l'image de
%$H^i_{\rm c}(\GG(\Q),{\cal C}(\GG(\A)))$ dans $H^i(\GG(\Q),{\cal C}(\GG(\A)))$,
On a une suite exacte longue
$$\xymatrix@R=4mm@C=4mm{
\cdots\ar[r]& H^i_{\rm c}(\GG(\Q),{\cal C}(\GG(\A)))\ar[r]& H^i(\GG(\Q),{\cal C}(\GG(\A)))\ar[d]\\
&& \bigoplus_{\PP\,{\rm mod}\sim}H^i(\PP(\Q),{\cal C}(\GG(\A)))\ar[r]&
H^{i+1}_{\rm c}(\GG(\Q),{\cal C}(\GG(\A)))\to\cdots}
$$

Soit $\PP=\UU\LL$ un parabolique de $\GG$, o\`u $\LL$ est le l\'evi de $\PP$
et $\UU$ est unipotent.  
%Disons que {\it la classe $X$ est acyclique pour $\PP$}
%si $H^0(\PP(\A),X(\PP(\A)))=L$, et si $H^i(\PP(\A),X(\PP(\A)))=0$ pour $i\geq 1$.
\begin{prop}\phantomsection\label{coco30}
%Si la classe $X$ est acyclique pour $\PP$,
On a un isomorphisme de repr\'esentations de $\GG(\A)$.
$$H^i(\PP(\Q),{\cal C}(\GG(\A)))={\rm Ind}_{\PP(\A)}^{\GG(\A)}H^i(\LL(\Q),{\cal C}(\LL(\A))),$$
o\`u $\PP(\A)$ agit sur $H^i(\LL(\Q),{\cal C}(\LL(\A)))$ \`a travers la surjection
naturelle $\PP(\A)\to \LL(\A)$.
\end{prop}
\begin{proof}
Pour all\'eger les notations, on note $\GG$, $\PP$, $\LL$, $\UU$,
les groupes $\GG(\A)$, $\PP(\A)$, $\LL(\A)$, $\UU(\A)$,
et ${\rm G}$, ${\rm P}$, ${\rm L}$, ${\rm U}$, les groupes
$\GG(\Q)$, $\PP(\Q)$, $\LL(\Q)$, $\UU(\Q)$.

On fait agir $b\in \PP$ sur $\PP\times \GG$ par
$b\cdot(x,y)=(xb,b^{-1}y)$, $\beta\in {\rm P}$ par $\beta\cdot(x,y)=(\beta^{-1}x,y)$
et $g\in\GG$ par $g\cdot (x,y)=(x,yg)$ (ces trois actions sont des actions \`a droite
et commutent deux \`a deux). Cela induit des actions \`a gauche sur ${\cal C}(\PP\times\GG)$.
On a 
$${\cal C}(\GG)={\rm Ind}_{\PP}^{\GG}{\cal C}(\PP)=H^0(\PP,{\cal C}(\PP\times \GG))$$
On en d\'eduit des isomorphismes
\begin{align*}
H^i\big({\rm P},{\cal C}(\GG)\big)
&\cong H^i\big({\rm P},H^0(\PP,{\cal C}(\PP\times \GG))\big)\\
&\cong H^i\big({\rm P}\times \PP,{\cal C}(\PP\times \GG)\big)\\ 
&\cong H^0\big({\PP},H^i({\rm P},{\cal C}({\PP}\times \GG))\big)\\ 
&\cong H^0\big(\PP,H^i({\rm P},{\cal C}(\PP))\wotimes {\cal C}(\GG)\big)\\
&\cong {\rm Ind}_{\PP}^{\GG}H^i({\rm P},{\cal C}(\PP)\big)
\end{align*}
En effet, le premier isomorphisme r\'esulte de ce qui pr\'ec\`ede, le dernier
est la d\'efinition de l'induite et l'avant-dernier r\'esulte de ce que
${\rm P}$ agit trivialement sur le facteur $\GG$ de $\PP\times\GG$. 
Pour \'etablir les deux isomorphismes restant,
il suffit de prouver que tous les autres termes des suites spectrales
convergeant vers $H^i({\rm P}\times \PP,{\cal C}(\PP\times \GG))$ sont nuls,
ce qui nous faisons apr\`es un petit interlude.
\begin{rema}
Soient $G$ un groupe topologique, $X$ une classe de fonctions\footnote{Pour la
prop.\,\ref{coco30}, telle qu'elle est \'enonc\'ee, 
on n'a besoin que de $X={\cal C}$, mais 
sa preuve s'adapterait pour d'autres $X$ (mais pas tous).}, et $M$ un $G$-module
tel que les orbites sont de classe $X$. On munit $X(G,M)$ de l'action
de $G$ donn\'ee par $(g\cdot\phi)(x)=g\cdot\phi(g^{-1}x)$.
Alors, $m\mapsto \alpha_m$, avec $\alpha_m(x)=x\cdot m$, est un isomorphisme
de $M$ sur $H^0(G,X(G,M))$.

Soit $\iota:X(G,M)\to X(G,H^0(G,X(G,M)))$ l'application d\'efinie par
$\iota(\phi)(z)=\alpha_{z^{-1}\cdot\phi(z)}$. 
Si on munit $X(G,H^0(G,X(G,M)))$ de l'action de $g\in G$ donn\'ee par 
$(g\cdot \Phi)(z)=\Phi(g^{-1}z)$, alors $\iota$ est un isomorphisme
$G$-\'equivariant, l'isomorphisme inverse s'obtenant
en composant les applications naturelles
$$X(G,H^0(G,X(G,M)))\hookrightarrow X(G,X(G,M))=X(G\times G,M)\to X(G,M)$$
o\`u la derni\`ere fl\`eche est $\phi^{(2)}\mapsto\phi$ avec $\phi(x)=\phi^{(2)}(x,x)$;
l'application compos\'ee $X(G,M)\to X(G\times G,M)$ est $\phi\mapsto\phi^{(2)}$,
avec $\phi^{(2)}(z,t)=tz^{-1}\cdot\phi(z)$.
Comme $X(G,H^0(G,X(G,M)))\cong X(G)\wotimes H^0(G,X(G,M))$,
cela fournit un isomorphisme $G$-\'equivariant
$$X(G,M)\cong X(G)\wotimes H^0(G,X(G,M))$$ 
o\`u $g\in G$ agit trivialement sur 
$H^0(G,X(G,M))$ et par $(g\cdot\phi)(x)=\phi(g^{-1}x)$ sur $X(G)$.
\end{rema}

Revenons \`a la preuve de la prop.\,\ref{coco30}.

$\bullet$ En tant que $\PP$-module,
${\cal C}(\PP\times\GG)\cong {\cal C}(\PP)\wotimes {\cal C}(\GG)$ (avec action
diagonale) est isomorphe \`a ${\cal C}(\PP)\wotimes {\cal C}(\GG)$ avec action triviale
sur ${\cal C}(\GG)$ et donc 
$$H^j(\PP,{\cal C}(\PP\times\GG))\cong
H^j(\PP,{\cal C}(\PP))\wotimes {\cal C}(\GG)$$
est nul si $j\geq 1$ puisque $H^j(\PP,{\cal C}(\PP))=0$ (lemme de Shapiro).
A fortiori, si $j\geq 1$,
$$H^{i-j}({\rm P},H^j(\PP,{\cal C}(\PP\times\GG)))=0$$

$\bullet$ De m\^eme, $H^{i-j}({\rm P},{\cal C}(\PP\times \GG))\cong
H^{i-j}({\rm P},{\cal C}(\GG)\wotimes {\cal C}(\PP))$, avec action diagonale de
$\PP$, est isomorphe \`a $H^{i-j}({\rm P},{\cal C}(\GG))\wotimes {\cal C}(\PP)$, avec action
triviale de $\PP$ sur $H^{i-j}({\rm P},{\cal C}(\GG))$. On en d\'eduit
que $$H^j(\PP, H^{i-j}({\rm P},{\cal C}(\GG))\wotimes {\cal C}(\PP))=0$$
 si $j\geq 1$
puisque $H^j(\PP,{\cal C}(\PP))=0$.

Pour conclure, il suffit donc de prouver que $H^i({\rm P},{\cal C}(\PP))\cong H^i({\rm L},{\cal C}(\LL))$.
Or un d\'evissage \`a partir du cas de ${\bf G}_a$ (prop.\,\ref{gagm3}) nous donne
$$H^i({\rm U},{\cal C}(\UU))=\begin{cases}L&{\text{si $i=0$,}}\\0&{\text{si $i\geq 1$.}}
\end{cases}$$
D'o\`u l'on d\'eduit que
$$H^i({\rm U},{\cal C}(\PP))=\begin{cases}{\cal C}(\LL)&{\text{si $i=0$,}}\\0&{\text{si $i\geq 1$.}}
\end{cases}$$
et on conclut en utilisant la suite exacte $1\mapsto {\rm U}\to {\rm P}\to {\rm L}\to 1$
et la suite spectrale associ\'ee \`a ce d\'evissage de ${\rm P}$.
\end{proof}

%\begin{prop}\label{coco31}
%Si $\LL$ est un tore d\'eploy\'e, alors
%\mar{Peut-\^etre c'est vrai en g\'en\'eral avec le r\'esultat sur $GL_1$}
%$$H^i(\LL(\Q),X(\LL(\A)))=\begin{cases}X(\LL(\wZ))&{\text{si $i=0$,}}\\0&{\text{si $i\geq 1$.}}
%\end{cases}$$
%\end{prop}
%\demo
%On utilise $\widehat\Gamma=\LL(\wZ)\times \LL(\R)_+$.
%Alors $\Gamma=\{1\}$, et le r\'esultat s'en d\'eduit.
%
%\begin{prop}\label{gener10}
%Si $\BB$ est un borel,
%\end{prop}
%

\end{document}